\documentclass{amsart}

\usepackage{graphicx}
\usepackage{amsmath}
\usepackage{amssymb}
\usepackage{amsfonts}
\usepackage{amscd}
\usepackage{amsthm}
\newtheorem{thm}{Theorem}
\newtheorem{prop}[thm]{Proposition}
\newtheorem{cor}[thm]{Corollary}
\newtheorem{lem}[thm]{Lemma}

\newtheorem{df}[thm]{Definition}
\newtheorem{rem}[thm]{Remark}
\newtheorem{eg}[thm]{Example}

\newtheorem{ques}[thm]{Question}

\newtheorem{alg}[thm]{Algorithm}

\title{Inverse problems in geometric graphs using internal measurements}
\author{Michael Robinson}

\begin{document}

\thanks{This work was supported under DARPA/STO HR0011-09-1-0050 and AFOSR FA9550-09-1-0643}

\begin{abstract}
This article examines the inverse problem for a lossy quantum graph
that is internally excited and sensed.  In particular, we supply an
algorithmic methodology for deducing the topology and geometric
structure of the underlying metric graph.  Our algorithms rely on
narrowband and visibility measurements, and are therefore of
considerable value to urban remote sensing applications.  In contrast
to the traditional methods in quantum graphs, we employ ideas related
to algebraic and differential topology directly to our problem.  This
neatly exposes and separates the impact of the graph topology and
geometry.
\end{abstract}

\maketitle

\section{Introduction}
This article opens a new line of inquiry into the structure of waves
on metric graphs.  In particular, we are interested in algorithms for
recovering the metric graph structure from measurements of propagating waves taken at
unspecified locations {\it within} the graph.  We recover this
structure in two stages: by first recovering the topology, and then the
geometry.  The algorithms for each stage draw upon a related circle of
ideas in algebraic topology; in particular, we use homological tools
to extract topology and cohomological tools to extract geometry.  To
validate the effectiveness of our approach, the algorithms have been
implemented in software and tested on simulation data. \cite{GHR_2010} 

\subsection{Historical development of the theory}
The study of linear
second-order differential operators on metric graphs has a long and
venerable history, which is discussed nicely in the survey articles
\cite{Kuchment_2008} \cite{Bandelt_2008} \cite{Gnutzmann_2006}, which
have extensive references.  The study of metric graphs
with self-adjoint linear operators began with a short paper by Pauling
\cite{Pauling_1936} examining the spectral properties of aromatic
compounds.  This paper was followed by a number of others that refined
Pauling's approach (for instance \cite{Ruedenberg_1953}, and of which
\cite{Richardson_1972} is a brief summary).  Because of the physical
chemistry focus, a metric graph paired with a self-adjoint
second-order linear operator became known as a {\it quantum graph}.
Because the spectum of a self-adjoint operator is what is measured in
quantum mechanics, the study of quantum graphs has been heavily
focused on their spectral properties.  The first mathematically
rigorous treatment of the spectrum of a quantum graph was by Roth in
\cite{Roth_1983}, which introduced the use of {\it trace formulae}.
Trace formulae permit the spectrum of a quantum graph to be treated as
a single algebraic object.

After Roth's initial work, an increasingly sophisticated literature on
spectral properties grew up with trace formulae as the primary tool.
As the spectral properties themselves are not of direct interest here,
we refer the reader to the excellent survey \cite{Gnutzmann_2006} and
will call out a few interesting articles to highlight the historical
development of the field.  Spacing between eigenvalues in regular and
equilateral graphs was treated by \cite{Jakobson_2003} and
\cite{Pankrashkin_2006}.  That the eigenvalues of a quantum graph are
generically simple was shown by \cite{Friedlander_2005}, making use of
a generalization of the theory of partial differential operators on
manifolds.  This simplifies some of the treatment of inverse spectral
problems.  Post \cite{Post_2008} examines some general properties
of the spectrum of quantum graphs, and relates them to combinatorial
Laplacians.  Finally, Parzanchevski \cite{Parzanchevski_2009}
showed that two different quantum graphs can have the same spectrum.

In addition to being interesting in its own right, Parzanchevski's
article \cite{Parzanchevski_2009} also demonstrates the kind of
limitations present in recovering the underlying metric graph from its
quantum graph spectrum.  In other words, the {\it inverse spectral
  problem} \cite{Carlson_1999} cannot be solved completely.  In line
with spectral methods, which concern globally-valid solutions to
linear equations on the quantum graphs, both inverse spectral problems
and inverse {\it scattering} problems concern the extraction of metric
graph structure from global measuments.  The global nature of theses
inverse problems is physically helpful, since quantum graphs have been
used to study very small objects: it is difficult to excite a single
atom in a molecule, usually one must settle for imposing global, {\it
  external} excitation.  

External excitations lead to scattering problems \cite{Mizuno_2008},
which in the case of quantum phenomena can lead to rather complicated
dynamics for particles confined to a graph \cite{Kottos_1997},
\cite{Kottos_2003}.  In particular, eigenfunctions of the Laplacian on
a quantum graph can have widely separated peaks that are irregularly
dispersed in the graph.  This permits particles to tunnel in an
apparently random fashion from one portion of the graph to another,
and has therefore been a good model of quantum chaos.  The examination
of general scattering problems leads to interesting questions strictly
outside of the quantum mechanical context.  For instance, Flesia {\it
  et al.}  \cite{Flesia_1987} examined how the excitation of a regular
graph with random speeds of propagation along it supports (or does not
support) localized excitations.  Indeed, they found that in classical
wave propagation, the peaks of eigenfunctions tend to be more widely
distributed than in the quantum case.

Inverse spectral problems on quantum graphs became popular following
the paper of Gutkin and Smilansky \cite{Gutkin_2001}, which was made
somewhat more precise by Kurasov and his collaborators
\cite{Kurasov_2001}, \cite{Kurasov_2005}.  Briefly, these articles
showed how to solve a generic class of inverse spectral problems when
the linear operator is the Laplacian.  The nongeneric problems clearly
must have symmetries, as \cite{Boman_2005} explains in detail, and
indeed Parzanchevski's examples of isospectral graphs are highly
symmetric.  For other operators, there are partial results, of which
\cite{Avdonin_2008} and \cite{Yurko_2005} are typical.  For compact
graphs in which the self-adjoint operator is the Laplacian, the
spectrum determines the total length of the graph (sum of all edge
lengths), the number of connected components, and the Euler
characteristic (and hence the number of loops) \cite{Kurasov_2008}.

Clearly, the topology of the underlying metric graph plays an
important role in both the spectral structure and the solution to
inverse problems.  That connection is of particular interest in this
article, as we show explicitly how algebraic topological invariants of
the graph impact the spectrum.  For instance, in \cite{Fulling_2007},
an index is computed that detects traveling wave solutions whose
orbits are closed but not periodic when the self-adjoint operator on
the graph is the Laplacian.  (By ``closed but not periodic,'' we mean
traveling wave solutions that return to their initial starting point,
but are going in the opposite direction at that starting point.)  This
index is topological, and depends only on the Euler characteristic of
the graph.  

\subsection{Main contributions of this article}
In this article, we look at quantum graphs from a substantially
different viewpoint, motivated by applications in remote sensing.  In
particular, while we are interested in inverse problems (finding the
metric graph structure from measurements), we are interested in
gathering {\it local} information only.  To fix ideas, consider a
dense urban environment, in which wave propagation can only occur in
narrow channels (along the roads), and is obstructed by many
buildings.  Within this environment are placed an unknown number of
transmitters and receivers at various locations.  Given that the
transmitters can be distinguished from one another by the receivers,
how much of the geometric and topological structure of the network of
propagation channels can be recovered?  Since the propagation channels
are narrow, it is sensible to assume that wave propagation happens
only along the edges of a graph.  The vertices of this graph represent
junctions between propagation channels in the usual way.  Of course,
the resulting second-order operator is the wave operator, which leads
to a quantum graph formulation of this problem.  (This ``thin-wire''
approximation has been examined asymptotically in
\cite{Kostrykin_1999}, \cite{Kuchment_2002}, \cite{Smilansky_2006},
and \cite{Molchanov_2006}.)  Unlike the inverse spectral and
scattering problems, the measurements are from {\it internal}
locations, rather than at a distance and concern local features of the
data.  Indeed, the continuity of the eigenfunctions is what will
permit us to infer global features from these local measurements.

There are three main reasons why the traditional quantum graph
analytic techniques would be impractical in a remote sensing
application:
\begin{enumerate}
\item In order to collect enough spectral data to exploit any of the
  inverse problem solutions found thus far, extremely wide-band
  receivers and transmitters would be needed.  In the case of urban
  remote sensing, the resulting frequencies are usually not available
  for imaging purposes.
\item Propagation losses can be substantial in remote sensing, which
  usually results in a strong localization of the signal to the
  vicinity of each transmitter.
\item The computational complexity of the algorithms resulting from
  the inverse spectral problems is unknown, and may hamper practical
  application.   
\end{enumerate}
These difficulties suggest that a better approach is to exploit
transmitter visibility information (ie. tabulate which transmitters
are visible by which receivers).  As will be shown, this permits the
topology of the underlying graph to be computed algorithmically and
efficiently.  We exhibit a novel algorithm for preprocessing
visibility information so that it is suitable for computing the
underlying topology.  Coherent, local measurements of the waves
permit the geometry to be recovered, under similar genericity
conditions to \cite{Gutkin_2001}.  Additionally, we obtain a new
characterization theorem that relates the space of eigenfunctions to
the geometric and topological structure of the graph.

The tools we employ are those in the emerging field of applied algebraic
topology.  In particular, we will employ the Nerve Lemma
\cite{Bott_1995} to determine the topology from visibility
information, the Whitney embedding theorem for stratified spaces
\cite{Natsume_1980}, and some ideas from the cohomology of
constructible sheaves.

It should be noted that the most similar approach that we have found
in the literature is that of the recent work by Caudrelier and Ragoucy
\cite{Caudrelier_2010}, which uses reflection-transmission algebras to
perform computations on internally-sampled signals.  We recover their
results using a sheaf-theoretic framework, which cements the
connection of their method to topological invariants.

\subsection{Outline of the article}
We begin in Section \ref{statment_sec} by giving a precise statement
of the problem we aim to solve, as well as stating the key results
that we prove in later sections.  In particular, we observe that there
exists (Theorem \ref{contractible_cover_thm}) a finite cover by
regions coming from thresholded signal levels.  This is an enabling
result for an algorithm (Algorithm \ref{refinement_alg}) for
extracting the graph topology from measurements.  We explain this
algorithm in Section \ref{topology_sec}.  In Section
\ref{geometry_sec}, we demonstrate that once the topology has been
found, narrowband signals suffice to compute the geometry (Theorem
\ref{geometry_extract_thm}, which has an inductive, algorithmic
proof).  In support of this result, we prove a characterization
(Theorem \ref{general_tcoh_thm}) that shows how the space of solutions
to quantum graphs is completely, and explicitly determined by the
geometry and topology of the graph.  Finally in Section
\ref{discussion_sec} we discuss the results.

\section{Problem statement and definitions}
\label{statment_sec}

\subsection{Review of metric graph definitions}

The underlying metric graph structure that will be used in this
article is essentially the one initially defined in \cite{Zhang_1993}
(and later examined in \cite{Baker_2006}, \cite{Baker_2007}), but to
support the examination of scattering problems we relax the
compactness hypothesis.

\begin{df}
A {\it metric graph} is a metric space $X$ such that each $x\in X$ has
a neighborhood $U_x$ isometric to a degree $n_x$ star-shaped set:
\begin{equation*}
\{z\in \mathbb{C} | z=te^{2\pi i k n_x}\text{ for some } 0\le t < r_x \text{ and
some }k \in \mathbb{Z} \}.
\end{equation*}
\end{df}

We can associate a combinatorial structure of edges and vertices
(nonuniquely) to a metric graph.  Precisely, if $V \subseteq X$ is a
discrete subset such that $V$ contains all points of $X$ whose degree
is not 2, we call $V(X)$ a set of {\it vertices}.  The set of
path-connected components of the complement of $V$ (each of which is
isometric to an open interval) is called the set of {\it edges}
$E(X)$.  Those elements of $V(X)$ that lie in the topological closure
of an edge $e$ are called the {\it endpoints} of $e$.  (There may be
only one endpoint in any given edge.)  (We remark that the choice of vertex and edge sets
will in general effect the quantum graph structures we obtain, but we
choose vertex conditions for which this nonuniqueness is
immaterial.)

We assume that it is always possible to find a {\it finite} set
of vertices and edges.  We call $X$ in this situation a {\it finite
  metric graph}.  With this structure, an edge that contains a degree
1 vertex will be called a {\it closed edge}.  An edge whose closure in
$X$ is not a circle, has exactly one endpoint, and has infinite length
will be called an {\it open edge}.

Of course, this assumption replaces the compactness assumption in
\cite{Zhang_1993} and others.

The primary contribution of \cite{Baker_2006} was the formulation of a
Laplacian operator suitable for finite metric graphs.  For a piecewise smooth
function $f:X\to \mathbb{C}$, this operator is the bounded, signed
{\it measure} given by 
\begin{equation}
\label{laplacian_eqn}
\Delta f(x) = \frac{d^2 f}{dx^2}(x)+\sum_{x\in X} \sigma_x(f) \delta_x,
\end{equation}
where $\sigma_x(f)$ is the sum of (outgoing) directional derivatives
of $f$ at $x$.\footnote{Caution: our sign convention differs from
  \cite{Baker_2006}.}  Notice in particular, that the second
derivative operator appearing in the first term of
\eqref{laplacian_eqn} is insensitive to the parametrization at $x$ on
the interior of an edge.  Secondly, the sum of directional derivatives
vanishes everywhere except at finitely many points since $f$ is
piecewise smooth.  Since the Laplacian is a measure, the fact that the second
derivatives are not defined at the vertices presents no difficulty.

\subsection{Lossy quantum graphs and their fundamental solutions}

Since we are interested in narrowband, lossy wave propagation, we
will study the fundamental solution to the Helmholtz equation
want to examine solutions to
\begin{equation}
\label{helmholtz_eqn}
\Delta u+k^2 u = \delta_y,
\end{equation}
where the wavenumber $k$ may be complex (to indicate lossy propagation), and $y\in X$
is the {\it transmitter location}.  We will be
interested in solutions that satisfy the ``Kirchoff''
conditions \cite{Kuchment_2008}:
\begin{itemize}
\item $u$ is continuous on $X$ and 
\item at each point $x \in X$, the sum of the derivatives
  of $u$, $\sigma_x(u)=0$.
\end{itemize}

\begin{df}
The pairing of \eqref{helmholtz_eqn} and the Kirchoff conditions will
be called a {\it lossy quantum graph}.
\end{df}

Clearly, along the interior of an edge, one has the general solution
to \eqref{helmholtz_eqn} which is the superposition of two traveling
waves:
\begin{equation}
\label{travel_eqn}
u(x)=c_1 e^{ikx}+c_2e^{-ikx}.
\end{equation}

The non-unique
choice of vertex and edge sets does not play a significant role in the
solutions of a lossy quantum graph when Kirchoff conditions are used.  Suppose $u$ is the solution to a
lossy quantum graph.  If we consider another quantum graph structure
in which there is an additional vertex within an edge, $u$ will
automatically be a solution of this new structure.  Likewise, any
solution of this new structure will be a solution of the original
structure.  This latter fact is most easily seen as a consequence of
edge collapse (Lemma \ref{tlcollapse}).

A fundamental solution \eqref{helmholtz_eqn} of a quantum graph
$X$ can be obtained as a superposition of solutions of the lossy quantum
graph on $X - \{y\}$.  It can be interpreted as an external source
scattering problem for the graph with $y$ removed. 

\subsection{Thresholding and covers of contractible regions}

In a lossy quantum graph, the energy from a signal source tends to be
concentrated around its location.  From an urban signal propagation
perspective, this means that only receivers nearby a given transmitter
will be able to decode its transmissions correctly.  This suggests
that {\it visibility} within a lossy quantum graph could
be useful for determining which receiver locations are near which
transmitters.  System designers typically use signal level
thresholds to determine whether decoding can proceed successfully, so
we employ a thresholding approach to construct visibility
regions for a given signal source's location.  Precisely, the region
where the signal level is above a given threshold will be a
contractible subset of $X$ when the threshold and signal loss are
large enough.

The central result is that there exists a collection of locations of
signal sources that can be thresholded appropriately to give a
collection of contractible visibility regions that cover
the graph.

\begin{thm}
\label{contractible_cover_thm}
Suppose $X$ is a lossy quantum graph with wavenumber
$k=k'+i\alpha$ for $k',\alpha>0$.  ($\alpha$ is the
{\it loss coefficient}, with larger values corresponding to higher
loss.)  
\begin{enumerate}
\item For each $y\in X$, there is a choice of $\alpha$ (say
  $\alpha_y$), and a threshold $T_y > 0$ such that for the fundamental
  solution $u_y$ to \eqref{helmholtz_eqn}, the {\it visibility
    region} $U_y(\alpha_y)=\{x \in X | |u_y(x)|^2 > T_y^2\}$ is
  contractible.
\item The collection $\{U_y(\alpha_y)\}$ forms a cover of $X$.
\item Suppose that all vertices of $X$ are contained in the compact
  set $K$.  There is therefore a finite subcollection
  $\{y_1,...,y_n\}$ such that
  $\{U_{y_1}(\alpha_{y_1}),...,U_{y_n}(\alpha_{y_n})\}$ covers $X \cap
  K$.  Let $\alpha = \max\{\alpha_{y_1},...,\alpha_{y_n}\}$.  Then the
  resulting collection $\{U_{y_1}(\alpha),...,U_{y_n}(\alpha)\}$
  covers $X \cap K$ and consists of contractible sets.
\end{enumerate}
\end{thm}

In order to prove Theorem \ref{contractible_cover_thm}, we need to
study the interference of two lossy traveling waves on a line
segment.  The requisite information is captured in the following Lemma.

\begin{figure}
\begin{center}
\includegraphics[width=3in]{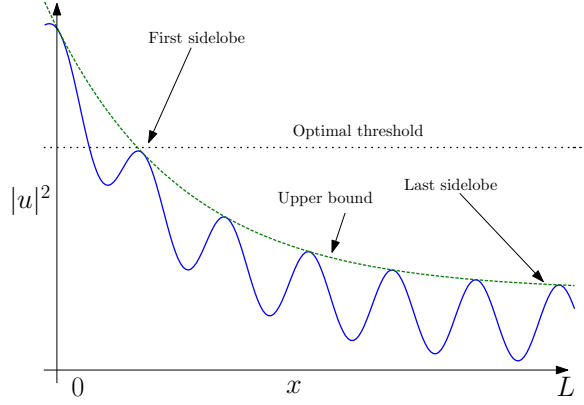}
\caption{Amplitude of interfering traveling waves}
\label{traveling_int_fig}
\end{center}
\end{figure}

\begin{lem}
\label{traveling_interference_lem}
Consider $0 \le x \le L$ and two traveling waves
$u(x)=e^{ikx}+ce^{-ik(x-L)}$ with $k=k'+i\alpha$ as in the statement of
Theorem \ref{contractible_cover_thm}.  Suppose that $c=\Gamma
e^{-\alpha L}$, where $\Gamma$ is the reflection coefficient at $x=L$.
Note that the amplitude of the right-going wave at $x=L$ is
$e^{-\alpha L}$.  For any fixed $\Gamma$, $k'$, and $L$, there exists
an $\alpha>0$ and a $T>0$ such that $\{x \in [0,L] | |u(x)|>T \}$ is
connected and contains zero.
\end{lem} 
Note that $T$ is strictly positive.
\begin{proof}
We proceed by computation...
\begin{eqnarray*}
u&=&e^{-\alpha x}e^{ik'x} + c e^{\alpha (x-L)} e^{-ik(x-L)}\\
&=&\left(e^{-\alpha x} \cos k'x + c e^{\alpha(x-L)}\cos k'(x-L)\right)
+ i\left(e^{-\alpha x} \sin k'x - c e^{\alpha(x-L)}\sin k'(x-L)\right).\\
\end{eqnarray*}
Thus, after a small amount of manipulation,
\begin{equation}
\label{traveling_interference_eqn1}
|u|^2=e^{-2\alpha x} + c^2e^{2\alpha (x-L)}+ 2ce^{-\alpha
  L}\cos\left(2k'x -k'L\right).
\end{equation}
Notice in particular that the fast spatial variation is captured by
the cosine term.  The location of the ``first sidelobe'' (see Figure \ref{traveling_int_fig}) of this
expression is the left most location of the maxima of the cosine,
namely
\begin{equation*}
x_{FSL}=\frac{n\pi}{2k'}+\frac{L}{2},
\end{equation*}
where $n$ is the smallest integer such that $x_{FSL}$ is positive.

An upper bound for $|u|^2$ is its envelope, namely
\begin{equation}
\label{traveling_interference_eqn2}
|u|^2\le e^{-2\alpha x} + c^2e^{2\alpha (x-L)}+ 2ce^{-\alpha
  L}.
\end{equation}

At the first sidelobe location, the value of $|u|^2$ is
\begin{eqnarray*}
|u(x_{FSL})|^2 &=& e^{-\alpha L}\left(e^{-\frac{\alpha n \pi}{k'}} +
  c^2e^{\frac{\alpha n \pi}{k'}}+2c\right)\\
&=&e^{-\alpha L}\left(e^{-\frac{\alpha n \pi}{k'}} +
  \Gamma^2e^{\frac{\alpha n \pi}{k'}-\alpha L }+2\Gamma e^{-\alpha L}\right),\\
\end{eqnarray*}
which can be made less than the value at 
\begin{equation*}
|u(0)|^2=1+\Gamma^2e^{-3\alpha L}+2\Gamma e^{-2\alpha L} \cos(-k'L)
\approx 1.
\end{equation*}
by taking $\alpha$ large enough.  Therefore, we'd like to take
$T=|u(x_{FSL})|$.  

Now the upper bound \eqref{traveling_interference_eqn2} is symmetric
about $x=\frac{L - \log c}{2}$, so the only thing that could spoil the
connectedness of the super-level set of $T$ is large values of $|u|$
near $L$.  Hence, we want 

\begin{eqnarray*}
L &<& \text{ the last sidelobe location before the minimal envelope}\\
&<&L-\log c - \frac{n\pi}{2k'} - \frac{L}{2} \text{ (an
  underestimate)}\\
\frac{L}{2}&<& - \log c - \frac{n\pi}{2k'}.\\ 
L\left(\frac{1}{2} -\alpha\right) & < & - \log \Gamma - \frac{n\pi}{2k'},
\end{eqnarray*}
which is clearly satisfied for large enough $\alpha$.
\end{proof}

\begin{rem}
The wavenumber $k'$ is fairly high for practical imaging systems
(roughly tens of waves per meter), so one typically spatially averages the
measurements to avoid aliasing.  This means that instead of
thresholding the amplitude \eqref{traveling_interference_eqn1}, we
threshold the envelope \eqref{traveling_interference_eqn2}.  This
results in a tolerance of lower loss, and also results in larger
visibility regions.
\end{rem}

Now, we can present the proof of Theorem \ref{contractible_cover_thm}:
\begin{proof}
\begin{enumerate}
\item This is immediate from repeated application of Lemma
  \ref{traveling_interference_lem}, and the conversion of each edge
  incident to a source $y$ into an externally excitated edge.  The
  source is at $y$, and the reflection coefficient $\Gamma$ captures
  effects from the rest of the graph as well.  Notice that more loss
  will make $\Gamma$ smaller, and thus a threshold that satisfies
  Lemma \ref{traveling_interference_lem} will continue to work.
\item That the visibility regions form a cover is immediate from the
  fact that each region obtained in Lemma \ref{traveling_interference_lem} contains zero.
\item The important point is that for larger $\alpha$, the sets
  $U_y(\alpha)$ can be increased in size, or even maintained at the
  same size (on a single edge) by the selection of a threshold.  In
  any event, it is apparent that the since the degree of any vertex is
  finite, a single threshold per source is easily found that results
  in cover by contractible sets even when the loss is increased. 
\end{enumerate}
\end{proof}

\begin{rem}
Theorem \ref{contractible_cover_thm} is an existence result only: many
configurations of transmitters will fail to provide a cover upon
thresholding.  However, by adding finitely many more transmitters to
areas of low transmitter density, a cover by contractible sets can be
obtained.
\end{rem}

\section{Topology computation}
\label{topology_sec}
In this section, we make use of the visibility regions of the previous
section (Theorem \ref{contractible_cover_thm}) to deduce the topology
of a quantum graph.  Our primary tool is the Nerve Lemma, generally
attributed to Leray \cite{Bott_1995}.  If one has a cover of the graph
by contractible sets (as in Theorem
\ref{contractible_cover_thm}), for which all finite intersections
between elements of the cover are also contractible, then one can
construct a simplicial space that is homotopy equivalent to the
original graph.  This {\it nerve construction} is easy to implement,
and for which there are computational topology tools available, for
instance \cite{Kaczynski_2004}.

Unlike the work of previous authors, the algorithm presented here
(in Section \ref{refinement_sec}) only relies on coarse
discriminations, rather than on a detailed examination of the spectrum.
In particular, the only assessment that is needed is whether the
intersection between a number of visibility regions is connected or
not.  It is not immediately clear that this determination is possible.
However, according to a generalization of Whitney's celebrated
embedding theorem \cite{Lee_2003}, there is a topology-preserving embedding of the
graph into an appropriately high-dimensional signal space, which we
construct.  The result is that if the intersection of several
visibility regions is disconnected, then its image is also
disconnected in the signal space.

The main idea of the proof of the algorithm's performance centers
around the Mayer-Vietoris sequence for homology, which is a
well-understood tool in algebraic topology.  We refer the reader to
the excellent treatment in \cite{Hatcher_2001} for details.

\subsection{Uniqueness of signal response and the detection of connected components}
\label{cluster_sec}

If the visibility regions found in Theorem \ref{contractible_cover_thm} formed a good cover, we would be able to use the Nerve Lemma to recover the homotopy type of the underlying metric graph.  However, the visibility regions will usually not form a good cover, often in the way shown in Figure \ref{obstacle_fig}.  However, in a graph, the intersection of two contractible sets can fail to be contractible only by being disconnected.  In this section, we show how to detect disconnectedness of sets by using three or more independent solutions to the quantum graph.

If there are $m$ signal sources, the complete set of measurements we
obtain at a receiver located at $x$ is the value of the piecewise smooth function
\begin{equation*}
P(x)=(u_1(x), u_2(x), ..., u_m(x)).
\end{equation*}
For brevity, we say that $P$ is the {\it signal profile}, taking locations in $X$ to {\it signal space} $S$.  

If the signal profile is topology-preserving (an embedding), then $U\subseteq X$ is disconnected if and only if its image under $P$ is disconnected.  Better, we could just look at the image (which are measurements taken from every point) and deduce the homeomorphism type of the graph.  (Indeed, homeomorphism type is stronger than homotopy type.)  However, this is unrealistic: normally we have to sample the receiver locations.  This is the reason for using a coarser construction based on {\it refining} the visibility regions into a good cover.

We can obtain the following result (appearing for {\it manifolds} in \cite{Golubitsky_1973}) using a stratified transversality result in \cite{Trotman_1979}.

\begin{thm}
\label{signal_embedding_thm}
If the dimension of the signal space $m$ (the number of transmitters) is greater than 2, the signal profile is generically (almost always) an embedding.
\end{thm}
\begin{proof}
First, recall that an embedding consists of an injective, piecewise smooth function whose derivative (where defined) is nonvanishing, by the inverse function theorem.  By direct computation along the edges, the derivative of a nontrivial solution to the Helmholtz equation is nonvanishing.  

It is also therefore the case that a metric graph is an (a)-regular Whitney stratified space.  ((a)-regularity is essentially trivial in this case since the only strata are vertices and edges, and vertices have trivial tangent space.  See \cite{Trotman_1979} for a precise definition.)

Rather than working with $P$ directly, we work with an augmented version
\begin{equation*}
F:C^2(X) \to C^2(X) \times S\times S,
\end{equation*}
given by
\begin{equation*}
F(x,y)=(x,y,P(x),P(y)),
\end{equation*}
where $C^2(X)=X\times X - \Delta_X$ is the space of all {\it nonoverlapping} pairs of points on the graph.  ($\Delta_X$ is the diagonal in $X$ and consists of pairs of points at the same location.)  We note that $C^2$ is an (a)-regular stratified space, being the product of two such spaces (removing the diagonal causes no difficulty).  Since $F$ is evidently an injective immersion, the image of $F$ is also (a)-regular.  (We note that without the factor of $C^2$ in the codomain, $F$ fails to be injective whenever $P$ is not injective; this disrupts the stratification of the image.)

$P$ fails to be injective if and only if the image of $F$ intersects $D=C^2\times \Delta_S$.  We use transversality to determine whether that intersection is empty or not.  The transversality result of \cite{Trotman_1979} states that $F$ is generically transverse to $D$.  (We actually perturb the submanifold $D$, not the image of $F$.  This amounts to perturbing not the recieved signals, but our ability to determine if the signals from two points agree.  We are looking for an empty intersection between the stratified set and the diagonal.  What we get then is stability of this empty intersection.)

When $F$ is transverse to $D$, then the intersection between any stratum of the image of $F$ and $D$ has dimension no greater than
\begin{eqnarray*}
\text{dim }C^2+\text{dim }D - \text{dim }(C^2\times S \times S)&=&2 \text{dim }X+(2\text{dim }X+m)-(2\text{dim }X+2m)\\
&=&2 \text{dim }X-m=2-m.
\end{eqnarray*}
When this bound on dimension is negative, there is no intersection under generic perturbations.  Hence, if the number of independent signals measured is greater than 2, the signal profile is generically injective, hence generically an embedding.
\end{proof}

\subsection{Coverage refinement}
\label{refinement_sec}

\begin{figure}
\begin{center}
\includegraphics[width=3in]{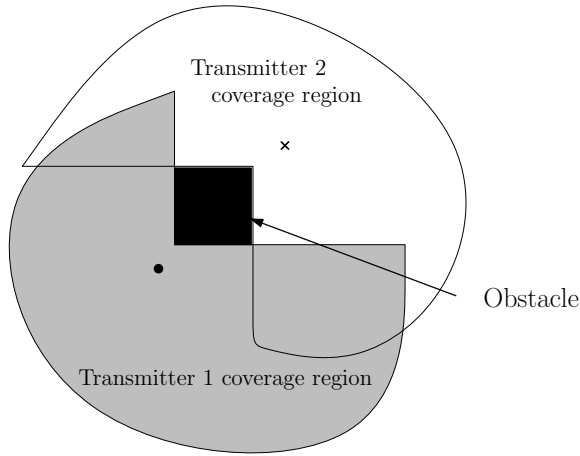}
\caption{Intersection of two contractible coverage regions may be disconnected.  (According to Theorem \ref{signal_embedding_thm}, another transmitter would be needed to disambigute the components of the intersection.)}
\label{obstacle_fig}
\end{center}
\end{figure}

The visibility regions associated to each transmitter can be made
contractible by appropriate thresholding.  However, this usually is
not enough to make the resulting cover into a ``good'' cover, and
thereby recover the topology.  In a graph, the only obstruction is
that intersections between coverage regions may be disconnected.  In
Section \ref{cluster_sec}, we showed that disconnected components
could be detected and identified from the received signals.  In this
section we exhibit an algorithm that exploits connected component
discrimination to {\it refine} the coverage regions into a good cover.

In this section, we use homology with a fixed field $\mathbb{F}$ as
the set of coefficients.  We will therefore suppress the coefficients
from the notation.  This has the advantage that in graphs, we can
identify acyclic sets with contractible ones \cite{Kaczynski_2004}.

\begin{thm}
\label{refinement_thm}
Suppose that $\mathcal{U}=\{U_1,U_2,...U_n\}$ are open sets forming a good cover of a subgraph of $X$ and that $W$ is a contractible subset of $X$ for which
\begin{itemize}
\item $W\not\subseteq \cup\mathcal{U}$, and
\item there are $U_{i_1}, U_{i_2}, ... \in \mathcal{U}$ such that $W\cap U_{i_1} \cap U_{i_2} \cap ...$ is disconnected. 
\end{itemize}
Let $V=W\cap (\cup\mathcal{U})$, then
\begin{enumerate}
\item each path component of $V$ is acyclic, and
\item there exist open neighborhoods $U'_i$ of $U_i - V$ for each $U_i \in \mathcal{U}$ and an open neighborhood $W'$ of $W-V$ that satisfy
\begin{enumerate}
\item $\cup\{U'_i\}$ and $W'$ are disjoint open sets,
\item each $U'_i$ and $W'$ consists of acyclic path components
, and
\item the collection of the path components of $V$, $W'$, and the $U'_i$ forms a good cover. 
\end{enumerate}
\end{enumerate}  
\end{thm}

It is straightforward to make use of Theorem \ref{refinement_thm} to
devise an iterative algorithm for refining a set of contractible
transmitter coverage regions $\mathcal{V}=\{V_1,V_2,...,V_n\}$.  We
have implemented this algorithm in software and tested it against
computer simulations. \cite{GHR_2010}
\begin{alg}
\label{refinement_alg}
We define good covers $\mathcal{V}_1,\mathcal{V}_2,...,\mathcal{V}_n$ inductively:
\begin{itemize}
\item Base step: $\mathcal{V}_1=\{V_1\}$ is a good cover by assumption.
\item Induction step: Assuming $\mathcal{V}_k$ is a good cover, we apply Theorem \ref{refinement_thm} with $W=V_{k+1}$ and $\mathcal{U}=\mathcal{V}_k$.  We then define $\mathcal{V}_{k+1}$ to be the good cover that we obtain as (2)(c) in Theorem \ref{refinement_thm}.
\end{itemize}
\end{alg}

We begin by proving a technical lemma that we primarily use to relate the homology of a graph to a subgraph.  This is an easy application of the Mayer-Vietoris sequence.

\begin{lem}
\label{refinement_lem}
Suppose that $f:X \to Y$ is a continuous injection from one graph into
another and that $Y-f(X)$ is a disjoint union of open, acyclic sets.
Then $f$ induces an injection $H_1(X)\to H_1(Y)$.
\end{lem}
\begin{proof} {\bf Simplify this proof!!}
First, observe that there is an open neighborhood $U$ of $f(X)$ that is homotopy equivalent to $f(X)$.  This neighborhood consists of the union of $f(X)$ and some subintervals of $Y-f(X)$ (taken from the ends of the edges in $Y-f(X)$ that are adjacent to $f(X)$).  Then we may consider the Mayer-Vietoris sequence
\begin{equation*}
0\to H_1(\text{some intervals})\to H_1(U)\oplus H_1(Y-f(X))\to H_1(Y)\to ...
\end{equation*}
since $U\cup\text{int }(Y-f(X))=Y$.  By assumption $U$ has the same homology as $f(X)$ and $X$, so this exact sequence can be written
\begin{equation*}
\begin{CD}
0\to H_1(\text{some vertices})\to H_1(X)@>f_*>>H_1(Y)\to ...
\end{CD}
\end{equation*}
A discrete space always has trivial $H_1$, whence $f_*$ must be injective on $H_1$.  
\end{proof}

The proof of Lemma \ref{refinement_thm} relies on the topological dimension of $X$ being not greater than 1.  Indeed, the Lemma is false in dimensions 2 and higher. 

We now proceed with the proof of Theorem \ref{refinement_thm}.

\begin{proof} 
Begin by considering the Mayer-Vietoris sequence 
\begin{equation*}
\begin{CD}
0\to H_1(V)\to H_1(\cup \mathcal{U})\oplus H_1(W)@>*>>H_1(W\cup(\cup \mathcal{U}))\to H_0(V)\to ...
\end{CD}
\end{equation*}
Note that $H_1(W)$ is trivial by assumption, and by Lemma \ref{refinement_lem} the map $*$ above is injective.  Hence each component of $V$ is acyclic, establishing (1).

Define $B_i=U_i \cap (\overline{V}-V)=U_i \cap \partial V$, which consists of the boundary points of $V$ that lie in $U_i$.  Similarly, define $C=W\cap \partial V$.

Observe that $C$ is disjoint from all of the $B_i$:
\begin{eqnarray*}
C \cap (\cup B_i) &=& (W\cap\partial V)\cap\left(\cup(U_i\cap\partial V)\right) \\
&=&W\cap(\cup U_i)\cap\partial V\\
&=&V\cap\partial V =0\\
\end{eqnarray*}
since $V$ is open.

Further, $\partial V$ is finite, since $X$ is a metric graph.

Define 
\begin{equation}
\label{refine_dist1_eqn}
\delta=\min_{x,y\in C\cup(\cup \{B_i\})\atop x\not=y} d(x,y).
\end{equation}
Since $C\cup(\cup \{B_i\})$ is finite and contains more than one point, $\delta>0$.  However, an upper bound for $\delta$ is
\begin{equation}
\label{refine_dist2_eqn}
\delta<\inf \{d(x,y)|x\in W-V,\;y\in(\cup\mathcal{U})-V\}.
\end{equation}

Define neighborhoods
\begin{equation*}
W'=(W-V)\cup\left(\bigcup_{x\in C} B_{\delta/3}(x)\right),
\end{equation*}
and for each $i$,
\begin{equation*}
U'_i=(U_i-V)\cup\left(\bigcup_{x\in B_i} B_{\delta/3}(x)\right),
\end{equation*}
where $B_r(x)$ is the open ball of radius $r$ centered at $x$.  Evidently by \eqref{refine_dist2_eqn}, $W' \cap U'_i=\emptyset$ for all $i$, which establishes (2)(a).  By \eqref{refine_dist1_eqn},
\begin{equation*}
W' \simeq W-V,\text{ and } U'_i \simeq U_i -V.
\end{equation*}

Notice that $H_1(W-V)$ and $H_1(U_i-V)$ are trivial by Lemma
\ref{refinement_lem}: consider the inclusion of $W-V \to W$.  We just proved that $V$ is a disjoint union of acyclic components, and it is evidently open.  This inclusion induces an injection on $H_1$.  However, $W$ is acyclic, so this implies that $H_1(W-V)$ must be trivial.  The same reasoning works for $U_i-V$, which
establishes (2)(b).

To show (2)(c), it remains to establish the following statements, which are immediate from the construction of $W'$ and $U'_i$:
\begin{itemize}
\item Components of $V$ and $W'$ have acyclic intersections since they're precisely half-open intervals,
\item Components of $V$ and $U'_i$ have acyclic intersections by the same logic, and
\item The collection of components of $\{U'_i\}$ form a good cover of $(\cup \mathcal{U})-V$.  Observe that disjoint intervals were attached to each $U_i$ in the construction of $U'_i$ such that if an interval $I_i$ was attached to $U_i$ and $I_j$ was attached to $U_j$, with $I_i \cap I_j \not= \emptyset$, then $I_i=I_j$.
\end{itemize}
\end{proof}

\section{Geometry computation}
\label{geometry_sec}

We now address the problem of obtaining the geometric structure of a
quantum graph by using a few narrowband measurements taken at points
internal to the graph.  The central result of this section shows how a
single solution over the entire graph determines the geometry up to
phase ambiguity.  It should be emphasized that this result is obtained
under the assumption that the topological structure of the underlying
metric graph is known, either as a homeomorphism type (which is
preferable), or more likely as a nerve.  In either case, we
assume that we are in possession of a 1-dimensional simplicial complex
that describes the underlying topological space structure of the
quantum graph.

Since the primary effect of geometry is on the phase of signals, the
measurement of geometric information in a quantum graph requires
coherent, time-sensitive processing.  From a signal processing
perspective, the resulting signal-to-noise requirements are
considerably more demanding than those required to obtain visibility
information.  As a result, the measurements of geometry in a lossy
quantum graph will tend to be useful over a fairly small region.  We
therefore compute geometry from local measurements, which are the
least contaminated by loss.  Additionally, we want to operate passively, without reference to known transmitter locations.  Therefore, without substantial loss of generality, we consider homogeneous solutions rather than
fundamental solutions.  

As an aside, the calculations given in this section are actually still
valid for {\it homogeneous} solutions of {\it lossy} quantum graphs.
These solutions correspond to waves incident from outside the graph.
(See the discussion following Question \ref{coh_ques} for details.)

The space of solutions to a given quantum graph is finite dimensional,
and depends on both topology and the geometry.  In Section
\ref{space_sec}, we compute this dimension from the graph structure.
In more traditional language, we obtain the dimension and structure of
the eigenspaces of the graph Laplacian operator.  (There is a delicate
interplay between loss and the dimension of these eigenspaces: signal
loss eliminates resonance phenomena that are extremely useful in
inverse spectral methods for detecting geometric and topological
features.)  {\it Sheaves} are the appropriate mathematical tool
for connecting local information to global behavior, so we develop {\it
  sheaf theoretic} computational tools for quantum graphs.  We obtain
sheaf-theoretic proofs of certain graph operations that preserve
quantum graph structure, which allows great computational
simplification.

We give a brief introduction to the key ideas of sheaf theory in
Section \ref{sheaf_intro_sec}.  For a more detailed exposition, we
refer the reader to Appendix 7 of \cite{Hubbard} and to \cite{Bredon}.
Building on this introductory material, we give an explicit definition
of the sheaf structure to be used on a quantum graph in Section
\ref{sheaf_def_sec}.  The global structure of this sheaf is computed
in Section \ref{space_sec} (Theorem \ref{general_tcoh_thm}), which
explicitly demonstrates the topological and geometric dependence of
the space of solutions.  As an aside, we note how this recovers an
inverse spectral result of previous authors.  Finally, in Section
\ref{geometry_extract_sec}, we show how combining the topology and a
single solution provides detailed geometric information about the
graph.

\subsection{A Brief introduction to sheaf theory}
\label{sheaf_intro_sec}

A sheaf is a mathematical tool for storing local information over a
domain.  It assigns some algebraic object, a vector space in our case,
to each open set, subject to certain compatibility
conditions.  These compatibility conditions are of two kinds: (1)
those that pertain to restricting the information from a larger to a
smaller open set, and (2) those that pertain to assembling information
on small open sets into information on larger ones.  What is of
particular interest is the relationship of the global information,
which is valid over the entire graph, to the topology of that graph.
Additionally, sheaf theory identifies classes of transformations of
the underlying graph that preserve the global information.  These
transformations permit us to simplify the graphs with no loss of generality.

\subsubsection{Elementary definitions for sheaves}
In this section, we follow the introduction to sheaves given in
Appendix 7 of \cite{Hubbard}, largely for its direct treatment of
sheaves over tame spaces.  For more a more general, and more
traditional approach, compare our discussion with \cite{Bredon}.

\begin{df}
A {\it presheaf} $F$ is the assignment of a vector space $F(U)$ to each
open set $U$ and the assignment of a linear map $\rho_U^V: F(U) \to F(V)$ for
each inclusion $V \subseteq U$.  We call the map $\rho_U^V$ the {\it
  restriction map} from $U$ to $V$.  Elements of $F(U)$ are called
{\it sections} of $F$ {\it defined over} $U$.  
\end{df}

\begin{df}
A {\it sheaf} $\mathcal{F}$ is a presheaf $F$ that satisfies the gluing axioms:
\begin{itemize}
\item (Monopresheaf) Suppose that $u \in \mathcal{F}(U)$ and that $\{U_1,
  U_2,...\}$ is an open cover of $U$.  If $\rho_U^{U_i} u = 0$ for
  each $i$, then $u=0$ in $\mathcal{F}(U)$.  Simply: sections that agree
  everywhere locally also agree globally.
\item (Conjunctivity) Suppose $u \in \mathcal{F}(U)$ and $v \in \mathcal{F}(V)$ are
  sections such that $\rho_U^{U\cap V} u = \rho_V^{U \cap V} v$.  Then
  there exists a $w \in \mathcal{F}(U \cup V)$ such that $\rho_{U \cup V}^U w =
  u$ and $\rho_{U \cup V}^V w = v$.  In other words, sections that
  agree on the intersection of their domains can be ``glued together''
  into a section that is defined over the union of their domains of
  definition.
\end{itemize}
\end{df}

\begin{eg}
Standard examples of sheaves are 
\begin{itemize}
\item The collection of continous real-valued functions
  $C(X,\mathbb{R})$ over a topological space $X$.
  In this case, the sections defined over an open set $U$ is $\{f:U
  \to \mathbb{R} | f \text{ is continuous}\}$.
\item The collection of locally constant functions, which essentially
  assigns a constant to each connected component of each open set.  
\end{itemize}

In contrast, the collection of {\it constant functions} does not form
a sheaf.  Suppose $u$ and $v$ are distinct constants defined over
disjoint sets $U$ and $V$.  The sheaf axioms would indicate that since
their domains of definition ($U$ and $V$) are disjoint, then there
should exist a constant function defined over $U \cup V$ that
restricts to each.  This is of course impossible, since such a
function is only locally constant.
\end{eg}

We now turn to the problem of understanding the effects of graph
operations on sheaves.  There are six famous operations on sheaves
that are important in the general theory, but only two of them
(cohomology and direct images) play a role in this article.

\subsubsection{Cohomology}
We can recast the conjunctivity axiom as testing if $\rho_U^{U\cap V}
u - \rho_V^{U \cap V} v$ is zero or not, rather than checking for
equality.  This can be viewed as looking at kernel of the linear map
$d:\mathcal{F}(U) \oplus \mathcal{F}(V) \to \mathcal{F}(U \cap V)$ given by $d(x,y)=\rho_U^{U\cap
  V}x-\rho_V^{U\cap V}y$.  Indeed, all of the elements of the kernel
of such a linear map correspond to the agreement of sections on $U
\cap V$.

On the other hand, the monopresheaf axiom indicates that the preimage
of zero under the map $d$ corresponds to the restriction of these
glued sections onto each of $U$ and $V$.  Indeed, any nonzero element
of the {\it image} of $d$ cannot be a section over $U \cup V$.  

These two points motivate a computational framework for working with
sheaves, called the \v{C}ech construction.  
\begin{df}
Suppose $\mathcal{F}$ is a sheaf on $X$, and that $\mathcal{U}=\{U_1,U_2,...\}$ is a cover
of $X$.  We define the {\it \v{C}ech cochain spaces}
$C^k(\mathcal{U};\mathcal{F})$ to be the direct sum of
the spaces of sections over each $k$-wise intersection of elements in
$\mathcal{U}$.  That is 
\begin{equation*}
C^k(\mathcal{U};\mathcal{F})=\bigoplus \mathcal{F}(U_{i_1} \cap
U_{i_2} \cap ... \cap U_{i_k}).
\end{equation*}
We define a sequence of linear maps 
\begin{equation*}
d^k:C^k(\mathcal{U};\mathcal{F}) \to C^{k+1}(\mathcal{U};\mathcal{F})
\end{equation*}
by 
\begin{equation*}
d^k(\alpha)(U_1,U_2,...,U_{k+1})=\sum_{i=0}^{k+1} (-1)^i \rho^{U_0
  \cap ... \hat{U}_i ... \cap U_{k+1}}_{U_0 \cap ... \cap U_{k+1}}
\alpha(U_0
  \cap ... \hat{U}_i ... \cap U_{k+1}),
\end{equation*}
where the hat means that an element is omitted from the list.
Note that these fit together into a sequence, called the {\it \v{C}ech
cochain complex}: 
\begin{equation*}
\begin{CD}
0 \to C^0(\mathcal{U};\mathcal{F}) @>d^0>> C^1(\mathcal{U};\mathcal{F}) @>d^1>> ...
\end{CD}
\end{equation*}
A standard computation shows that $d_k \circ d_{k-1} = 0$, so that we can define
the {\it $k$-th \v{C}ech cohomology space}
\begin{equation*}
\check{H}^k(\mathcal{U};\mathcal{F}) = \text{ker } d_k / \text{image } d_{k-1}.
\end{equation*}
\end{df}

The $\check{H}^k$ apparently depend on the choice of cover
$\mathcal{U}$, but for {\it good} covers (much as in the Nerve Lemma),
this dependence vanishes. Leray's theorem for sheaves states that
$\check{H}^k(\mathcal{U};\mathcal{F})$ is the same for each good
cover.  Indeed, it then depends only on $X$.  So we write
$H^k(X;\mathcal{F})=\check{H}^k(\mathcal{U};\mathcal{F})$, the sheaf's
{\it cohomology} in the case that $\mathcal{U}$ is a good cover.

A little thought about good covers on graphs reveals two important facts:
\begin{itemize}
\item for a metric graph $X$, $H^k(X;\mathcal{F})=0$ for $k>1$, and 
\item $H^0(X;\mathcal{F})$ is isomorphic to the space of global sections $\mathcal{F}(X)$.  
\end{itemize}

The first fact follows immediately from a covering dimension argument.
The latter fact comes from our construction, and that the image of
$d^{-1}$ ({\it not} $d$ inverse!) is zero.  This suggests a
computational way to examine solutions to quantum graphs by way of
sheaf theory: we construct a sheaf that has local sections being the
{\it locally valid} solutions to the quantum graph, and then compute
its cohomology.  What is especially valuable about this approach is
that there is {\it additional} information in $H^1(X;\mathcal{F})$.
This additional information plays a useful role in Section
\ref{space_sec} and detects {\it resonance} phenomena in quantum
graphs.

By analogy with the Mayer-Vietoris sequence for homology, there is a
Mayer-Vietoris sequence for sheaf cohomology.  It is given by the
following theorem (also in \cite{Bredon} in a variety of forms):

\begin{thm}
Suppose that $A,B$ are two open subspaces of a graph $X$ that
cover $X$, and that $\mathcal{F}$ is a sheaf over $X$.  
Then the following {\it Mayer-Vietoris sequence} is an exact sequence:
\begin{equation*}
\begin{CD}
...\to H^k(X;\mathcal{F})@>r>>H^k(A;\mathcal{F})\oplus H^k(B;\mathcal{F})@>d>>H^k(A \cap B;\mathcal{F})@>\delta>>H^{k+1}(X;\mathcal{F})\to...
\end{CD}
\end{equation*}
In this sequence, $r$ comes from restriction maps in the obvious way,
$d$ is the composition of restriction maps and a difference: $d(x,y)=
\rho_A^{A\cap B} x - \rho_B^{A\cap B} y$, and $\delta$ is the
connecting homomorphism.  Note that notation has been abused above
slightly: by $H^k(A;\mathcal{F})$ we mean the $k$-th cohomology of the sheaf $\mathcal{F}$
restricted to subsets lying in $A$.
\end{thm}

\subsubsection{Direct images}

\begin{df}
Suppose $f:X \to Y$ is a continuous function between topological
spaces.  If $\mathcal{F}$ is a sheaf over $X$, then the {\it direct image} of
$\mathcal{F}$ through $f$ is written $f_* \mathcal{F}$ and is given by its action
\begin{equation*}
(f_* \mathcal{F})(U)= \mathcal{F}(f^{-1}(U))
\end{equation*}
on open sets $U \subseteq Y$.  Clearly, $f_* \mathcal{F}$ is a sheaf over $Y$.
\end{df}

The key question is when the cohomologies of $\mathcal{F}$ and $f_* \mathcal{F}$ agree.
If they do agree, and the combinatorial structure of $Y$ is much
simpler than $X$, then it is usually much easier to compute the
cohomology of $f_* \mathcal{F}$ rather than of $\mathcal{F}$.  We make extensive use of
this in the sequel.  This essential question is answered by the
following result:

\begin{thm}(The Vietoris mapping theorem, Theorem 11.1 in Chapter II of \cite{Bredon}) Suppose that $f:X \to Y$ is
  a closed, continuous function between metric graphs, and $\mathcal{F}$ is a sheaf over
  $X$.  If $H^k(f^{-1}(\{y\});\mathcal{F})$ for $k>0$ and each $y \in Y$, then
  $\mathcal{F}$ and $f_* \mathcal{F}$ have isomorphic cohomology.
\end{thm}

\subsection{Definition of a sheaf structure for quantum graphs}
\label{sheaf_def_sec}

The collection of locally-valid, continuous solutions to a
differential equation on $\mathbb{R}$ forms a sheaf: restricting the
domain of validity of a local solution produces a new local solution,
and local solutions can be joined if they agree on the intersection of
their domains of validity.  This construction also works on a quantum
graph, resulting in the sheaf of {\it excitations}.  

Since solutions on a single edge consist of two propagating waves, we
obtain a more explicit definition for the sheaf of excitations.  We
construct the sheaf of excitations as a direct image of a sheaf over a
directed graph with two oppositely-oriented edges for each undirected
edge of our original quantum graph. \cite{Kuchment_2008} Our
construction permits a slight generalization over the sheaf of
excitations in that the values of the solution may take values in any
field rather than the complex numbers.  This simplifies the notation,
so we perform all further computations at this level of generality.
The resulting general object is called a {\it transmission-line
  sheaf}.

\begin{prop}
Local solutions of a quantum graph form a sheaf, called the {\it excitation sheaf}.
\end{prop}

Sheaves provide a convenient way to collate solutions to differential
equations.  For instance, \cite{Spencer_1969} discusses the situation
in great detail for systems of linear partial differential equations.
We sketch a proof here for our specific instance.

\begin{proof}
The appropriate sheaf is a correspondence between open sets (of the
topological space structure of the graph) and the space of continuous
functions over them, in which inclusion of open sets corresponds to
restriction of functions.  Solutions on an open set in the interior of
an edge remain solutions when restricted to a smaller domain.
Additionally, if solutions on two open sets within the interior of an
edge agree (pointwise) on the intersection of these two open sets,
then there is a solution defined on the union of these open sets.

The vertex conditions require continuity of solutions and some
additional constraints.  Suppose that $s$ is a solution defined on
$U$, an open set containing a vertex $v$.  Clearly, $s$ satisfies the
vertex conditions at $v$.  Now if $V \subset U$ is another open set,
restricting $s$ to $V$ will result in another solution.  In
particular, if $v \in V$, then the $s|V$ will still satisfy the vertex
condition at $v$.  Now if $s_1$ and $s_2$ are solutions defined on
$U_1$ and $U_2$, and $v$ lies in $U_1 \cap U_2$, then both $s_1$ and
$s_2$ will satisfy the vertex conditions at $v$.  If $s_1|(U_1 \cap
U_2)=s_2|(U_1 \cap U_2)$, then we can define a new solution $s$
pointwise by
\begin{equation*}
s(x)=\begin{cases}
s_1(x) & \text{ if }x \in U_1\\
s_2(x) & \text{ if }x \in U_2\\
\end{cases}
\end{equation*}
This new solution clearly satisfies the vertex condition at $v$, and is otherwise still at solution.
\end{proof}

The excitation sheaf is a complicated object, one in which the
combinatorial structure is obscured.  In particular, it is unclear how
to compute its sheaf cohomology, which encapsulates all of its
associated global structure.  Therefore, we construct a new sheaf
called a {\it flow sheaf} with simpler structure that has isomorphic
cohomology.  The decisive feature of this flow sheaf is that it is
{\it constructible} with respect to the graph structure.  This means
that over edges or vertices, the sheaf structure is
constant. \cite{Schurmann} Only where edges attach to vertices does
the sheaf structure change, and then only in a constrained fashion.

A flow sheaf, however, is defined over a directed graph.  We'll use a
direct image construction to write the appropriate sheaf definition
for quantum graphs, which are undirected.

\begin{df}
\label{mflow_df}
Suppose that $\mathbb{F}$ is a field and that $X$ is a directed graph.
Given that $X$ has the usual topology, let $\mathcal{U}=\{U_\alpha,V_\beta\}$ be a base for the topology of $X$ where each $U_\alpha$ is connected and contains exactly on vertex and each $V_\beta$ is contained in the interior of a single edge.
An {\it $\mathbb{F}$-flow sheaf on $X$} is the sheafification of the following presheaf $F$, defined on $\mathcal{U}$:
\begin{itemize}
\item $F(U_\alpha)$ is a direct sum of copies $\mathbb{F}$, one for each incoming edge into the unique vertex contained in $U_\alpha$,
\item $F(V_\beta)=\mathbb{F}$,
\item if $V_\beta \subset U_\alpha$ and $V_\beta$ is contained in the $n$-th incoming edge, the the restriction map $F(U_\alpha) \to F(V_\beta)$ is projection onto the $n$-th copy of $\mathbb{F}$,
\item if $V_\beta \subset U_\alpha$ and $V_\beta$ is contained in the $n$-th outgoing edge, then there is a fixed $\mathbb{F}$-linear map $F(U_\alpha) \to F(V_\beta)$ depending only on the vertex $v$ contained in $U_\alpha$ and $n$ (the outgoing edge).
This collection of maps, one for each outgoing edge, is called the {\it local coding at $v$} and is denoted by $\phi_n(v)$.
\end{itemize}
\end{df}

The following proposition is immediate from the definition of a constructible sheaf.
\begin{prop}
Flow sheaves are constructible with respect to the skeleta of $X$.
\end{prop}

Now let us consider the more specific case of a direct image of a flow
sheaf that is isomorphic to the excitation sheaf defined earlier.
Suppose $\mathcal{E}_k$ is an excitation sheaf with wavenumber $k$
over a quantum graph $X$.  Let $V$ and $E$ denote the vertex and edge
sets of $X$, respectively.  Denote by $L_e$ the length of edge $e \in
E$.

Construct a new graph $Y$ (that is a directed metric graph), by
replacing each undirected edge $v_1\leftrightarrow v_2$ of $X$ by a pair of
opposing edges $v_1\to v_2, \; v_2\to v_1$.  Let $f:Y \to X$ be the map
that sends each such pair of directed edges to the corresponding
undirected edge of $X$, as shown in Figure \ref{tlsheaf_fig}.  We
construct a $\mathbb{C}$-flow sheaf $\mathcal{F}$ over $Y$ with coding
maps tailored to work like the Kirchoff conditions, as follows.
Consider a vertex $v \in V$, with degree $n$ in $X$.  Label the
incoming edges (in $Y$) $\{a_1,a_2,...a_n\}$ and the outgoing edges
$\{b_1,b_2,...b_n\}$ so that $f(a_i)=f(b_i)$ for each $i$.  Define the
coding map by
\begin{equation}
\label{kirchoff}
\phi_i(v)(z_1,z_2,...z_n)=\frac{2}{n}\left( \sum_{j=1}^n e^{\sqrt{-1} kL_{a_j}} z_j \right) - e^{\sqrt{-1} kL_{a_i}}z_i,
\end{equation}
where $z_i$ is the value of a section of $\mathcal{F}$ restricted to the $i$-th incoming edge

\begin{figure}
\begin{center}
\includegraphics[width=2.5in]{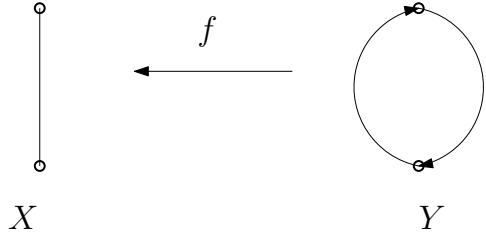}
\caption{Construction of a transmission line sheaf on an undirected graph from a flow sheaf on a directed graph}
\label{tlsheaf_fig}
\end{center}
\end{figure}

We call $\mathcal{G}_k=f_* \mathcal{F}$ a {\it transmission line sheaf}.
Any sheaf constructed over an undirected graph as the direct image of an $\mathbb{F}$-flow 
sheaf with coding maps satifying \eqref{kirchoff} (but with $\mathbb{F}$ endomorphisms 
specified on the edges, rather than phases) will also be called an $\mathbb{F}$-transmission line sheaf.

Suppose $\mathcal{F}$ is a transmission line sheaf on a graph $X$.  A {\it resonant} edge is one in which the edge endomorphisms are identity maps.  If the edge endomorphisms differ from the identity, we call the edge {\it nonresonant}.

\begin{prop}
The excitation sheaf $\mathcal{E}_k$ and the transmission line sheaf
$\mathcal{G}_k$ are isomorphic as sheaves.  As a result, we may study
the solutions to a quantum graph by computing the cohomology of the
transmission line sheaf instead.
\begin{proof}
One need only observe that the Kirchoff conditions are equivalent to the coding map \eqref{kirchoff} employed at the vertices of $X$, by a straightforward algebraic transformation.
\end{proof}
\end{prop}

\begin{figure}
\begin{center}
\includegraphics[width=3in]{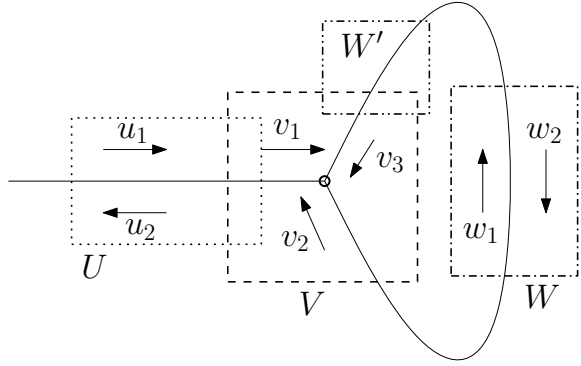}
\caption{Example of a transmission line sheaf}
\label{tlexample_fig}
\end{center}
\end{figure}

\begin{eg}
Consider the quantum graph shown in Figure
\ref{tlexample_fig}.  This graph has one vertex, one open edge, and
one loop edge.  Assume that the length of the loop is $L$.  We 
examine the sheaf on three different open sets, $U$, $V$, and $W$ as
shown in the figure.  The dimension of the space of sections over $U$
and $W$ is 2, and the dimension of the space of sections over $V$ is
3.  In the figure, basis elements (in terms of a traveling wave
decomposition) are shown.  

If we consider sections that are defined
over $U \cup V$, they must agree on $U \cap V$.  This induces the
following gluing conditions:
$u_1 = v_1$ and 
$u_2 = \frac{1}{3} v_1 + \frac{2}{3} e^{ikL} v_2 + \frac{2}{3} e^{ikL} v_3.$

Observe that since $V$ and $W$ are disjoint, there are no gluing
conditions required to construct sections over $V \cup W$.  If
instead, we moved $W$ to $W'$, so that $V \cap W' \not= \emptyset$,
two gluing conditions apply: $w_1 = v_3$ and
$w_2 = \frac{2}{3} v_1 + \frac{2}{3} e^{ikL} v_2 + \frac{1}{3} e^{ikL} v_3.$

\end{eg}

\subsection{Computation of the space of solutions}
\label{space_sec}

The objective of this section is to completely characterize the
cohomology of transmission line sheaves in terms of the geometry and
topology of the underlying metric graph.  The main result is the
following theorem:

\begin{thm}
\label{general_tcoh_thm}
Suppose $\mathcal{G}$ is an $\mathbb{F}$-transmission line sheaf with
on a connected quantum graph that has $l$ closed edges (of which
$l'$ of them have edge endomorphisms $E$ that satisfy $E-E^{-1}=0$, a
resonance condition for closed edges),
$m$ open edges, and $n$ resonant loops.  Then
\begin{equation}
\text{dim }H^0(X;\mathcal{G})=
\begin{cases}
n+1&\text{if }l=m=0\\
n+m&\text{if }l=0,m \not= 0\\
n+1+\min\{0,l'-1\}&\text{if }l\not= 0,m=0\\
n+m+\min\{0,l'-1\}&\text{otherwise}\\ 
\end{cases},
\end{equation}
\begin{equation}
\text{dim }H^1(X;\mathcal{G})\cong
\begin{cases}
n+1&\text{if }l=m=0\\
n&\text{if }l=0,m \not= 0\\
n+1+\min\{0,l'-1\}&\text{if }l\not= 0,m=0\\
n+\min\{0,l'-1\}&\text{otherwise}\\ 
\end{cases}.
\end{equation}
\end{thm}

This rather long section begins first with a discussion of some of the
implications of this result in Section \ref{obstruct_sec}.  The proof
of Theorem \ref{general_tcoh_thm} is outlined in Section
\ref{tcoh_pf_sec}, with the proofs of the key edge collapse lemmas
being postponed until Section \ref{collapse_sec}.

\subsubsection{Sheaf obstruction theory for quantum graphs}
\label{obstruct_sec}
Of course, Theorem \ref{general_tcoh_thm} strongly restricts the
geometry of the underlying graph, but it also restricts the topology.
In this section, we therefore also answer the question:

\begin{ques}
\label{coh_ques}
Given the cohomology of a transmission line sheaf $\mathcal{F}$ on a
quantum graph $X$, how much does it restrict the topology
of $X$?
\end{ques}

If there is loss in the quantum graph then no loop and no closed edges
are resonant.  Therefore, in lossy quantum graphs, the cohomology is
entirely determined by the number of open edges and whether there are
any closed edges (resonant or not).  The space of excitations of a
lossy quantum graph is essentially determined by external scattering
effects only.

Clearly, a graph with no resonant loops and no resonant edges will
have trivial $H^1$.  The converse is not true, though.  Closed edges
will typically induce a trivial $H^1$ if considered on their own,
regardless of edge endomorphism.  This means that the cohomology of
$\mathbb{F}$-transmission line sheaves is not a strong enough
invariant to distinguish between graphs that contain closed edges and
those that do not.  In particular, consider a graph $X$ that consists
of one vertex and $n$ open edges, and a graph $Y$ that consists of one
vertex, $n$ open edges, and one closed edge.  Regardless of the edge
endomorphism for the closed edge in $Y$, any transmission line sheaf
over $X$ will have the same cohomology as any other transmission line
sheaf over $Y$, though clearly their geometry differs.  This is a
sheaf-theoretic expression of the existence of isospectral quantum
graphs \cite{Parzanchevski_2009}, but goes a little farther: there
exist isospectral graphs whose eigenspaces agree as well.

A convenient way of summarizing the above discussion is that
nontrivial sheaf cohomology classes constitute {\it obstructions} to
the topology of the graph being trivial.

\begin{lem}
\label{obstruct_lem_1}
Suppose that $\mathcal{F}$ is an $\mathbb{F}$-transmission line sheaf
on a metric graph $X$ which has no closed edges.  Then
\begin{itemize}
\item the number of open edges in $X$ is equal to the Euler
  characteristic of $\mathcal{F}$, namely $\chi(\mathcal{F})=\text{dim
  }H^0(X;\mathcal{F})-\text{dim }H^1(X;\mathcal{F})$, and
\item there are at least $\text{dim }H^1(X;\mathcal{F})$ loops in $X$.
\end{itemize}
\end{lem}

The appearance of the Euler characteristic in Lemma
\ref{obstruct_lem_1} fits neatly with the interpretation of $H^1$ as
describing resonances.  In particular, one can think of the number of
open edges being the total number of excitations ($\text{dim
}H^0(X;\mathcal{F})$), which includes both externally-induced
excitations and internally-induced excitations (resonances) minus the
number of resonances ($\text{dim }H^1(X;\mathcal{F})$).  Therefore,
{\emph the cohomology of the excitation sheaf over a metric graph
  encompasses both scattering problems and eigenvalue problems.}

If we consider the more specific case of lossless excitation sheaves, then the
exact number of loops can be determined.  Suppose $X$ is a finite
metric graph with no closed edges.  As noted in Lemma
\ref{obstruct_lem_1}, the number of open edges can be computed from
the cohomology of any excitation sheaf.  Each edge of the graph has a
(set of) resonant frequencies, some of which may coincide.
Generically, none of the resonant frequencies coincide, and it is
therefore easy to locate the lowest resonant frequency of each edge,
which determines the length of each edge.  This is essentially the
idea of \cite{Gutkin_2001}.  However, the genericity condition can be
relaxed, since the cohomology of excitation sheaves captures edge
multiplicity information.

This discussion can be summarized by the following proposition and algorithm.

\begin{prop}
\label{topology_obstruct}
The topology and edge lengths of a finite metric graph with no closed edges is completely determined by the cohomologies of the excitation sheaves over it.
\end{prop}

\begin{alg}
Suppose that $X$ is a finite metric graph with no closed edges, and the $\mathcal{E}_k$ is the excitation sheaf on $X$ with wavenumber $k$.  
\begin{enumerate}
\item Determine the number of open edges $m$ using Lemma \ref{obstruct_lem_1}.
\item Define 
\begin{equation*}
S_0=\begin{cases}
\{(k,p)|k\in\mathbb{R}^+\text{ and dim }H^1(X;\mathcal{E}_k)=p\not= 0\} & \text{if }m\not= 0\\
\{(k,p)|k\in\mathbb{R}^+\text{ and dim }H^1(X;\mathcal{E}_k)=p-1\not= 0\} & \text{if }m=0\\
\end{cases}
\end{equation*}
This is the set of resonant wavenumbers of loops in $X$, counted with multiplicity.
\item Compute $k_i=\min \{k| \text{there exists a }p\in \mathbb{N} \text{ such that } (k,p)\in S_i\}$.  This is a fundamental resonant wavenumber of an edge in $X$.  Notice that this is well-defined since $S_i$ is countable.
\item Define 
\begin{equation*}
S_{i+1}=\{(k,p)\in S_i | k\notin k_i\mathbb{N}\} \cup 
\{(k,p-1)|(k,p)\in S_i \text{ and } k\in k_i\mathbb{N}\text{ and } p>1\}
\end{equation*}
\item Iterate steps 3 and 4.  Notice that since $X$ is finite, the $S_i$ stabilize at the empty set after finitely many iterations.
\end{enumerate}
\end{alg}

\subsubsection{Proof of Theorem \ref{general_tcoh_thm}}
\label{tcoh_pf_sec}

The proof relies on three edge collapse results (Lemma
\ref{mcollapse}, Lemma \ref{tlcollapse}, and Lemma
\ref{tlcollapse_nonresonant}) that permit combinatorial simplification
of the quantum graph without disrupting the structure of the
solutions.  Combinatorial edge collapse is not new, and plays an
important role in the quantum graphs literature, for instance
\cite{Kuchment_2008}.  These lemmas permit a direct, explicit
computation of the cohomology of transmission line sheaves. (For a
similar computational methodology, which finds a natural incarnation
as the sheaf theory we mention here, see \cite{Caudrelier_2010}.)

\begin{df}
Suppose that $X$ and $Y$ are finite metric graphs, which may be disconnected.
A map $f:X \to Y$ is called an {\it edge collapse} if
\begin{enumerate}
\item There exists an edge $e$ of $X$ for which $f(e)$ is a vertex of $Y$,
\item $f$ restricted to $X-e$ is a homeomorphism onto its image, which is $Y-f(e)$.
\end{enumerate}
\end{df}

\begin{figure}
\begin{center}
\includegraphics[width=4in]{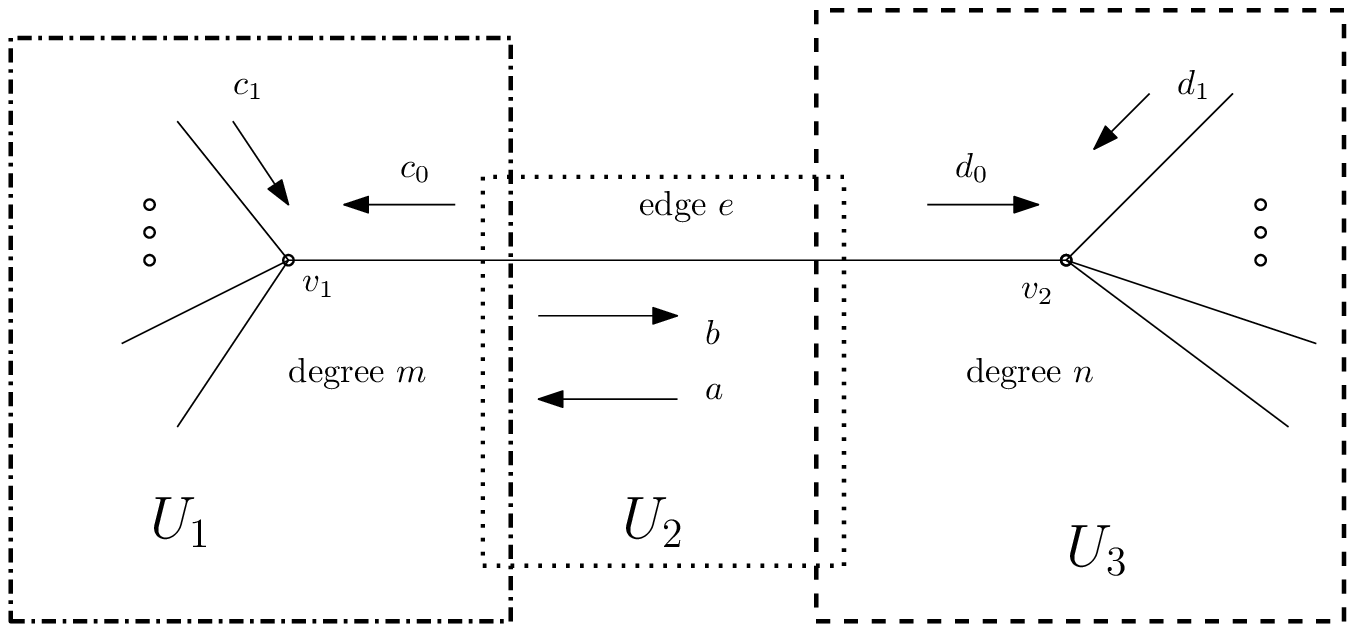}
\caption{Sets that cover the edge $e$ which is to be collapsed}
\label{tlcollapse_fig}
\end{center}
\end{figure}

The first edge collapse lemma applies to a flow sheaf on a directed
graph, and is the most general result for these kind of sheaves.

\begin{lem}
\label{mcollapse}
Suppose that $f:X \to Y$ is an edge collapse of an edge $e$ that has distinct endpoints, and
that $\mathcal{F}$ is an $\mathbb{F}$-flow sheaf on $X$. Then $f$ induces an isomorphism on cohomology:
\begin{equation}
H^*(X;\mathcal{F}) \cong H^*(Y;f_* \mathcal{F}),
\end{equation}
and $f_* \mathcal{F}$ is an $\mathbb{F}$-flow sheaf on $Y$ in which the coding maps at the vertices in $Y-f(e)$ are unchanged from those in $X-e$.
The coding map at $e$ is given by the following.  Let $v_1$ and $v_2$ be the endpoints of $e$, in which $e$ is incoming for $v_2$. 
Without loss of generality, suppose that $e$ is the first output of $v_1$ and the first input to $v_2$.  Then
\begin{equation}
\phi_i(f(e))(a_1,a_2,...a_m,b_2,b_3,...b_p)=
\begin{cases}
\phi_{i-1}(v_1)(a_1,a_2,...a_m)&\text{ if }i<n-1\\
\phi_{i-n}(v_2)(\phi_1(v_1)(a_1,a_2,...a_m),b_2,b_3,...b_p)&\text{otherwise}\\
\end{cases}
\end{equation}
where the input degree of $v_1$ is $m$, the output degree of $v_1$ is $n$, and the input degree of $v_w$ is $p$.  See Figure \ref{tlcollapse_fig} for a graphical representation of this situation.
\end{lem}

It is of course tempting to wonder if multiple edges between a pair of vertices can be collapsed together, something which would be a bit stronger than Lemma \ref{mcollapse}.  The success of this depends delicately on the coding maps.
In the case of transmission line sheaves, edge collapse works as desired.

\begin{lem}
\label{tlcollapse}
Suppose that $f:X\to Y$ is an edge collapse for an edge $e$ with distinct endpoints, each of which have degree greater than 1.
Then for any transmission line sheaf $\mathcal{G}$ with field coefficients, $f$ induces an isomomrphism on sheaf cohomology and the direct image $f_* \mathcal{G}$ is also a transmission line sheaf.
\end{lem}

Of course, Lemma \ref{tlcollapse} specifically excludes the case of
collapsing an edge loop, since such an edge does not have distinct
endpoints.  The final edge collapse lemma permits loops to be
collapsed provided they are not resonant.  

\begin{lem}
\label{tlcollapse_nonresonant}
Suppose $\mathcal{F}$ is an $\mathbb{F}$-transmission line sheaf over a graph $X$, and $f:X \to Y$ is a map between graphs that collapses a nonresonant loop.  Then $f$ induces an isomorphism on cohomology.
\end{lem}

With these lemmas stated (we prove them in Section
\ref{collapse_sec}), we now address the problem of proving Theorem
\ref{general_tcoh_thm}.  

\begin{figure}
\begin{center}
\includegraphics[width=3in]{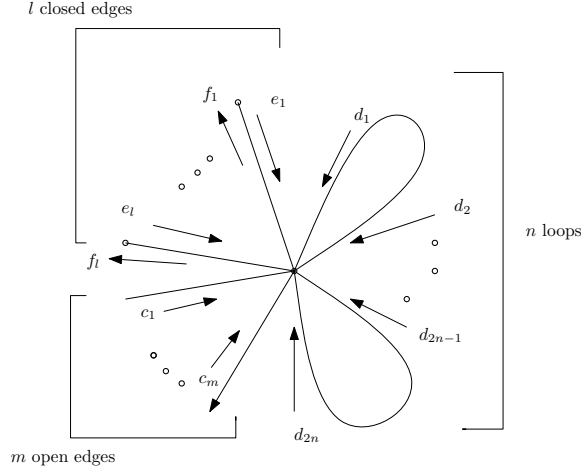}
\caption{A graph with one vertex, $l$ closed edges, $m$ open edges, and $n$ loops}
\label{tlcoh_1_fig}
\end{center}
\end{figure}

\begin{figure}
\begin{center}
\includegraphics[width=3in]{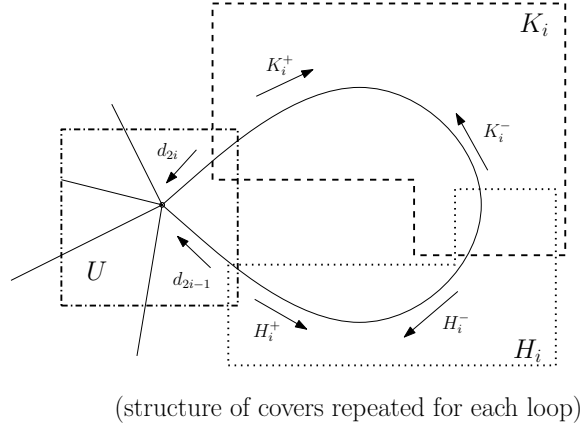}
\caption{Cover over a particular loop in the graph}
\label{tlcoh_2_fig}
\end{center}
\end{figure}

\begin{proof} (of Theorem \ref{general_tcoh_thm})
Begin by collapsing out all nonresonant loops using Lemma
\ref{tlcollapse_nonresonant}.  Then, obtain a minimal spanning tree
for $X$, and collapse each edge in the tree using Lemma
\ref{tlcollapse}.  Thus the sheaf cohomology of $\mathcal{G}$ can be
computed by computing the sheaf cohomology on a graph $Y$ with a single
vertex, $l$ closed edges, $m$ open edges, and $n$ loops as shown in
Figure \ref{tlcoh_1_fig}.

Choose a good cover for this space by selecting $U$ as a
contractible open set containing the single vertex of $Y$.  To each
loop $j$, associate two additional open sets $K_j$, $H_j$ homeomorphic to
intervals, which complete the cover as shown in Figure \ref{tlcoh_2_fig}.  

Note then that the \v{C}ech
cochain complex is 
\begin{equation}
\begin{CD}
0 @>>> \mathbb{F}^{m+2n} \oplus \mathbb{F}^{4n} \oplus \mathbb{F}^{2l} @>\delta >>
\mathbb{F}^{6n} \oplus \mathbb{F}^{2l} @>>> 0.
\end{CD}
\end{equation}
For convenience, let us consider the case where $l=0$, and let $d=m+2n$, which is the degree of the single
vertex of $Y$.  Organize the coboundary map $\delta$ so that it has
the following block form (columns are $c_1,...,c_m,d_1,d_2,H_1^+,H_1^-,K_1^+,K_1^-,...$):
\begin{equation*}
\begin{pmatrix}
D_{6 \times m}&B_{6 \times 6}&A_{6 \times 6}&...&A_{6 \times 6}\\
D_{6 \times m}&A_{6 \times 6}&B_{6 \times 6}&...&A_{6
  \times 6}\\
D_{6 \times m}&A_{6 \times 6}&A_{6 \times 6}&...&A_{6
  \times 6}\\
 &&&...&&\\
D_{6 \times m}&A_{6 \times 6}&A_{6 \times
  6}&...&B_{6 \times 6}\\
\end{pmatrix}
\end{equation*}
where
\begin{equation*}
A_{6\times 6}=\begin{pmatrix}
0&0&0&0&0&0\\
\frac{2}{d}&\frac{2}{d}&0&0&0&0\\
\frac{2}{d}&\frac{2}{d}&0&0&0&0\\
0&0&0&0&0&0\\
0&0&0&0&0&0\\
0&0&0&0&0&0\\
\end{pmatrix},
\end{equation*}
\begin{equation*}
B_{6\times 6}=\begin{pmatrix}
1&0&-1&0&0&0\\
\frac{2-d}{d}&\frac{2}{d}&0&-1&0&0\\
\frac{2}{d}&\frac{2-d}{d}&0&0&-1&0\\
0&1&0&-1&0&0\\
0&0&1&0&-1&0\\
0&0&0&1&0&-1\\
\end{pmatrix},
\end{equation*}
\begin{equation*}
D_{6\times m}=\begin{pmatrix}
0&0&...&0\\
\frac{2}{d}&\frac{2}{d}&...&\frac{2}{d}\\
\frac{2}{d}&\frac{2}{d}&...&\frac{2}{d}\\
0&0&...&0\\
0&0&...&0\\
0&0&...&0\\
\end{pmatrix}.
\end{equation*}
We can do row-reduction on each $B_{6 \times 6}$ block individually,
and in such a way that this preserves all other entries not in this
block (in particular, the nonzero rows remain unchanged outside the
block).  We obtain a new block 
\begin{equation*}
B'_{6 \times 6}=\begin{pmatrix}
1&0&-1&0&0&0\\
\frac{2-d}{d}&\frac{2-d}{d}&0&0&0&0\\
\frac{2-d}{d}&\frac{2-d}{d}&0&0&0&0\\
0&1&0&0&0&-1\\
0&0&1&0&-1&0\\
0&0&0&1&0&-1\\
\end{pmatrix},
\end{equation*}
which has rank 4 if $d=2$ and rank 5 otherwise.  Now observe that for
the nonzero entries of the second rows of each $B'$ block.  They are
\begin{equation*}
\begin{matrix}
\frac{2-d}{d}&\frac{2}{d}&\frac{2}{d}&...&\frac{2}{d}\\
\frac{2}{d}&\frac{2-d}{d}&\frac{2}{d}&...&\frac{2}{d}\\
\frac{2}{d}&\frac{2}{d}&\frac{2-d}{d}&...&\frac{2}{d}\\
&&...&&\\
\frac{2}{d}&\frac{2}{d}&\frac{2}{d}&...&\frac{2-d}{d}\\
\end{matrix}.
\end{equation*}
In each column, there are $n-1$ copies of $2/d$ and one copy of
$\frac{2-d}{d}$, so each column sums to $\frac{2n}{d}-1$.  Now when
$n=\frac{d}{2}=\frac{m+2n}{2}$, ie. $m=0$, this implies that the
coboundary matrix has an additional kernel element, and clearly at
most one such element.  On the other hannd, if $n \not= d/2$, clearly
all of these second rows of each $B'$ block are linearly independent.
Hence 
\begin{equation*}
\text{rank } \delta = \begin{cases}
4 & \text{if }d=2\\
5n & \text{if }m \not= 0\\
5n-1 & \text{if }m=0\\
\end{cases}.
\end{equation*}

Now if $l>0$, so there are closed edges present, the coboundary map
must be augmented with $2l$ rows and columns.  Let $d=2n+m+l$.  For the $i$-th closed edge,
we add two new rows, which look like
\begin{equation*}
\begin{matrix}
0&0&...&1&-E_i^{-1}&...&0&...m\text{ copies}...&0&0&...2n\text{ copies}...&0\\
\frac{2}{d}&0&...&\frac{2-d}{d}&E_i&...&\frac{2}{d}&...m\text{ copies}...&\frac{2}{d}&\frac{2}{d}&...2n\text{ copies}...&\frac{2}{d}\\
\end{matrix},
\end{equation*}
in which the first $2l$ columns are added, the
next $m$ columns are the first $m$ columns of the original $\delta$,
and the remaining columns count off in pairs from the first two
columns of each $B_{6 \times 6}$ block.  More precisely, they
correspond to the columns labelled $e_1, f_1,..., e_i, f_i, ..., e_l, f_l, 
c_1, ..., c_m, d_1, ..., d_{2n}$.  After one row operation, this becomes
\begin{equation*}
\begin{matrix}
0&0&...&1&-E_i^{-1}&...&0&...m\text{ copies}...&0&0&...2n\text{ copies}...&0\\
\frac{2}{d}&0&...&\frac{2}{d}&E_i-E_i^{-1}&...&\frac{2}{d}&...m\text{ copies}...&\frac{2}{d}&\frac{2}{d}&...2n\text{ copies}...&\frac{2}{d}\\
\end{matrix}.
\end{equation*}
Note that we thereby obtain a duplicate copy of the second row for
each closed edge, so at most 1 is contributed to the rank by these
rows.  One the other hand, the first row is clearly linearly
independent from all the others.  Hence the rank of the coboundary map
is increased by $\max \left\{ 2l, 2l-l'+1 \right\}$. 
\end{proof}

\subsubsection{Proofs of the edge collapse lemmas}
\label{collapse_sec}

We begin by addressing the most general edge collapse result, Lemma
\ref{mcollapse}.  The central difficulty is that the edge
endomorphisms and coding maps are not specified with a particular
form.  This complicates the calculations somewhat.

\begin{proof} (of Lemma \ref{mcollapse})
We aim to employ the Vietoris mapping theorem to obtain the desired isomorphism on cohomology.
To this end, observe that since $f$ is an edge collapse, it follows that it is a closed surjection.
Additionally, $X$ and $Y$ are both paracompact, so $f^{-1}$ is always taut.
Suppose that $y\in Y$, and discern two cases:
\begin{enumerate}
\item That $y$ is not $f(e)$, in which case $f^{-1}$ is exactly one point, so $H^p(f^{-1}(y);\mathcal{F})=0$ for $p>0$.
\item If $y=f(e)$, observe that $H^p(e;\mathcal{F})=\varinjlim H^p(U_\alpha;\mathcal{F})$, where $U_\alpha$ ranges over open sets
containing $e$.  We consider a good cover of $U_\alpha$ that consists of $V_1$ (containing the vertex $v_1$) and $V_2$ (containing $v_2$) whose intersection lies in the interior of $e$.
The \v{C}ech complex is then
\begin{equation}
\begin{CD}
0 @>>> M^m \oplus M^p @> \delta >> M @>>> 0
\end{CD}
\end{equation}
In which the coboundary map $\delta$ is given by $(a_1,a_2,...a_m,b_1,b_2,...b_p) \mapsto \phi_1(v_1)(a_1,a_2,...a_m) - b_1$.
Since $\phi_1(v_1)$ is a homomorphism, it's clear that the image of $\delta$ is $\mathbb{F}$. 
Hence, $H^p(e;\mathcal{F})=0$ for $p>0$.
\end{enumerate}

For the second statement, observe that the only thing to check is that the stalk over $f(e)$ has the correct rank.
In this case, that rank is $m+p-1$, which agrees with Definition \ref{mflow_df}.
\end{proof}

\begin{rem}
The formula for the coding map at the collapsed vertex $f(e)$ can be written in terms of
matrices (with entries in $\mathbb{F}$).  Suppose $A$ is a matrix for $\phi(v_1)$ and $B$ is a matrix for $\phi(v_2)$.
Let $u^T$ be the first row of $A$, which corresponds to the output of $v_1$ along the edge $e$.
Likewise, let $v$ be the first column of $B$, which corresponds to the input to $v_2$ coming from $e$.
Let $a$ and $b$ be the matrices obtained by deleting the first row of $A$ and first column of $B$, respectively.
Then $\phi(f(e))$ has block matrix form
\begin{equation}
\phi(f(e))=\begin{pmatrix}
a&0\\vu^T&b\\
\end{pmatrix}.
\end{equation}
\end{rem}

\begin{cor}
If $\mathcal{F}$ is an $\mathbb{F}$-flow sheaf over a finite, connected graph, its sheaf cohomology can be computed by looking at the direct image under the collapse of a spanning tree.
\end{cor}

For the proof of Lemma \ref{tlcollapse}, in addition to verifying that
the Vietoris mapping theorem still holds, we must also verify that the
direct image is still a transmission line sheaf.  This requires a
straightforward, but lengthy, computation.

\begin{proof}(of Lemma \ref{tlcollapse})
We need only redo the case of $H^*(f^{-1}(f(e));\mathcal{F})$ in Lemma
\ref{mcollapse}.  In this case, look at a good cover of the edge $e$,
which consists of three sets $\{U_1,U_2,U_3\}$.  Let $U_1$ contain one
endpoint of $e$ (degree $m$), $U_3$ contain the other (degree $n$),
and $U_2$ lie entirely within the interior of $e$.  We'll assume that the $e$ has an edge endomorphism $L$. 

In the \v{C}ech cochain complex, the coboundary map has the form
\begin{equation}
\label{cbndry_tlcollapse}
\begin{pmatrix}
1&0&...&0&0&0&...&0&-L^{-1}&0\\
\frac{2-m}{m}&\frac{2}{m} &...& \frac{2}{m}&0&0&...&0&0&-1\\
0&0&...&0&\frac{2-n}{n}&\frac{2}{n}&...&\frac{2}{n}&-1&0\\
0&0&...&0&1&0&...&0&0&-L\\
\end{pmatrix},
\end{equation}
which we claim has rank 2.  (The columns are organized by $\mathcal{G}(U_1)\oplus \mathcal{G}(U_3)\oplus \mathcal{G}(U_2)$.) Hence the $H^1(e,\mathcal{G})=0$.

We now address the claim by examining the kernel of the coboundary map.  By a pair of linear combinations of rows of \eqref{cbndry_tlcollapse}, we obtain
\begin{equation*}
d_0=L\left(\frac{2}{m}\sum_{j=0}^{m-1}c_j-c_0\right),
\end{equation*}
and
\begin{equation*}
c_0=L^{-1}\left(\frac{2}{n}\sum_{j=0}^{n-1}d_j-d_0\right).
\end{equation*}
Then we can solve for $d_0$ by substitution:
\begin{eqnarray*}
d_0&=&L\left(\frac{2}{m}\sum_{j=1}^{m-1}c_j-\left(\frac{2}{m}-1\right)c_0\right)\\
&=&L\frac{2}{m}\sum_{j=1}^{m-1}c_j-\left(\frac{2}{m}-1\right)\left(\frac{2}{n}\sum_{j=0}^{n-1}d_j-d_0\right)\\
&=&L\frac{2}{m}\sum_{j=1}^{m-1}c_j-\frac{2}{n}\left(\frac{2}{m}-1\right)\sum_{j=1}^{n-1}d_j-\left(\frac{2}{m}-1\right)\left(\frac{2}{n}-1\right)d_0\\
\end{eqnarray*}
\begin{eqnarray*}
\left(1-\left(\frac{2}{m}-1\right)\left(\frac{2}{n}-1\right)\right)d_0&=&\frac{2L}{m}\sum_{j=1}^{m-1}c_j+\frac{2}{n}\left(\frac{2-m}{m}\right)\sum_{j=1}^{n-1}d_j\\
\left(\frac{2m+2n-4}{mn}\right)d_0&=&\frac{2L}{m}\sum_{j=1}^{m-1}c_j+\frac{2}{n}\left(\frac{2-m}{m}\right)\sum_{j=1}^{n-1}d_j,\\
\end{eqnarray*}
whence
\begin{equation}
\label{d0_tlcollapse_coding}
d_0=\frac{nL}{m+n-2}\sum_{j=1}^{m-1}c_j+\frac{2-m}{m+n-2}\sum_{j=1}^{n-1}d_j.
\end{equation}
In particular, this confirms that the rank of \eqref{cbndry_tlcollapse} is 2, since the kernel is of dimension $m+n-2$.  

Continuing with \eqref{d0_tlcollapse_coding}, we show that the direct image is a transmission line sheaf by exhibiting a typical restriction map to an output of one of the edges.  Without loss of generality, we consider the output along the $i$-th edge of $v_2$, namely
\begin{eqnarray*}
d_i'&=&\frac{2}{n}\sum_{j=0}^{n-1}d_j-d_i\\
&=&\frac{2L}{m+n-2}\sum_{j=1}^{m-1}c_j+\frac{2}{n}\left(\frac{2-m}{m+n-2}\right)\sum_{j=1}^{n-1}d_j+\frac{2}{n}\sum_{j=1}^{n-1}d_j-d_i\\
&=&\frac{2}{m+n-2}\left(\sum_{j=1}^{m-1}Lc_j+\sum_{j=1}^{n-1}d_j\right)-d_i,
\end{eqnarray*}
which is of the form required for a transmission line sheaf.
\end{proof}

\begin{figure}
\begin{center}
\includegraphics[width=3in]{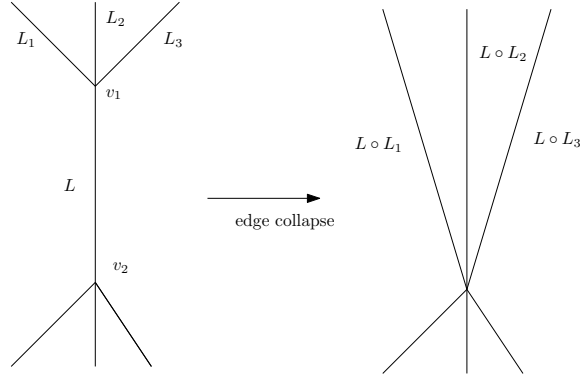}
\caption{Edge collapse results in edge endomorphisms being composed}
\label{tlcollapse_composition}
\end{center}
\end{figure}

\begin{rem}
Notice that when the edge $e$ is collapsed using Lemma \ref{tlcollapse}, the edge endomorphisms on one side are composed (see Figure \ref{tlcollapse_composition}).  This composition is non-unique: in the proof, we composed $L$ with the edges on the side of $v_1$, but we could have composed by $L^{-1}$ on the side of $v_2$.  This clearly results in quasi-isomorphic sheaves, as they will agree in cohomology, and this difference will not concern our discussion here.  
\end{rem}

\begin{rem}
\label{closed_rem1}
We note that a degree 1 vertex causes the proof of Lemma \ref{tlcollapse} to fail because the direct image is generally not a transmission line sheaf.  The Vietoris Mapping theorem applies perfectly well in this case, but the coding maps do in fact change.  It is true that $H^1(e)$ is still trivial, even if there is a nonidentity edge endomorphism.

Suppose for instance that $m=1$, and that we wish to compute the restriction to the output of the $i$-th edge incident to $v_2$.  This has the value
\begin{equation*}
d_i'=\frac{2}{n}\sum_{j=0}^{n-1}d_j-d_i.
\end{equation*}
If it the collapse of the edge resulted in a transmission line sheaf, then we should have 
\begin{equation*}
d_i'=\frac{2}{n-1}\sum_{j=1}^{n-1}d_j-d_i
\end{equation*}
upon eliminating $d_0$ using the conditions at $v_1$.  However, what we instead obtain is
\begin{equation*}
d_i'=\frac{2}{n}\left(\frac{L+L^{-1}}{(2/n-1)L+L^{-1}}\right)\sum_{j=1}^{n-1}d_j-d_i,
\end{equation*}
where $L$ is the edge endomorphism for $e$.  Although certain values for $L$ will result in a transmission line sheaf, generic values of $L$ will not.

An interpretation of this result is that the degree 1 vertex's influence is to adjust the coding maps, essentially ``tuning'' the transmission line.  In the case where a transmission line sheaf results from an edge collapse of a closed edge, the edge is the correct length to have no effect at all, which is related to resonance phenomena.
\end{rem}

Finally, we address the case of Lemma \ref{tlcollapse_nonresonant}.
Rather than using the \v{C}ech approach as in the earlier
calculations, we instead use a Mayer-Vietoris sequence, to illustrate
an alternate technique for cohomology computation.  This has the
advantage of requiring fewer dimensions, and is a fairly natural
context to consider the computation of sheaf cohomology over graphs.

\begin{figure}
\begin{center}
\includegraphics[width=3in]{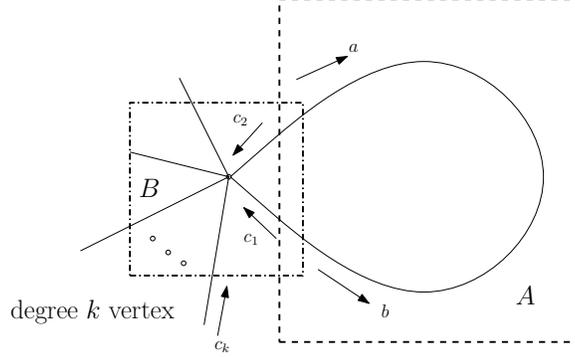}
\caption{Cover of a loop in the graph, for use with the Mayer-Vietoris sequence}
\label{tlcollapse_nonresonant_fig}
\end{center}
\end{figure}

\begin{proof} (of Lemma \ref{tlcollapse_nonresonant})
Let $U$ be a connected open set covering the loop to be collapsed, and contain exactly one vertex of degree $k$.  We form a cover of $U$ by two open sets $A$ and $B$.  Let $A$ be a connected open set contained in the interior of the edge, and $B$ be a connected open set containing the unique vertex in $U$ (see Figure \ref{tlcollapse_nonresonant_fig}).  Observe that $\mathcal{F}(A) \cong \mathbb{F}^2$ and $\mathcal{F}(B) \cong \mathbb{F}^k$.  The appropriate Mayer-Vietoris sequence is 
\begin{equation*}
\begin{CD}
0 \to H^0(U;\mathcal{F}) \to H^0(A;\mathcal{F}) \oplus H^0(B;\mathcal{F}) @>\Delta>> H^0(A \cap B;\mathcal{F}) @>\delta>> H^1(U;\mathcal{F}) \to 0,
\end{CD}
\end{equation*}
where we observe that $H^0(A \cap B;\mathcal{F}) \cong \mathbb{F}^4$.  As a result of this exact sequence, $H^0(U;\mathcal{F}) \cong \ker \Delta$, and $H^1(U;\mathcal{F}) \cong \text{image } \delta$.  Observe that

\begin{eqnarray*}
\text{dim ker }\delta &=& 4 - (k+2) + \text{dim ker }\Delta\\
&=&2-k+\text{dim }H^0(U;\mathcal{F}).\\
\end{eqnarray*}
Suppose that the edge endomorphism for the edge to be collapsed is $L:\mathbb{F} \to \mathbb{F}$.  Then 
\begin{equation*}
\Delta \left(\begin{pmatrix}a\\b\end{pmatrix},\begin{pmatrix}c_1\\...\\c_k\end{pmatrix}\right)=\begin{pmatrix}c_2-La\\c_1-L^{-1}b\\a-\frac{2}{k}(c_1+...+c_k)+c_1\\b-\frac{2}{k}(c_1+...+c_k)+c_2\end{pmatrix},
\end{equation*}
using Kirchoff conditions at the vertex.  Performing some algebraic manipulations, we find that elements of the kernel of $\Delta$ satisfy (provided $k\not= 2$)
\begin{equation*}
a=-\left(1-\frac{2}{k}L\right)^{-1}\left(1-\frac{2}{k}\right)L^{-1}b+\left(1-\frac{2}{k}L\right)^{-1}\frac{2}{k}(c_3+...+c_k)
\end{equation*}
and
\begin{equation*}
0=\left(1-\frac{2}{k}L^{-1}\right)b+\left(1-\frac{2}{k}\right)La-\frac{2}{k}\left(c_3+...+c_k\right).
\end{equation*}
If $k\not= 2$, then this leads to
\begin{equation*}
0=(2-L-L^{-1})b+(L-1)(c_3+...+c_k),
\end{equation*}
which means that the dimension of the kernel of $\Delta$ is $k-2$ if $L \not= 1$ and $k-1$ if $L=1$.

If instead, $k=2$, then we obtain
\begin{equation*}
a=La,\;b=L^{-1}b,
\end{equation*}
which implies that the dimension of the kernel of $\Delta$ is 0 if $L \not= 1$ and 2 otherwise.

Thus we have
\begin{equation*}
H^0(U;\mathcal{F}) \cong \begin{cases}
\mathbb{F}^{k-1} & \text{if } L=1 \text{ and }k \not= 2\\
\mathbb{F}^{k-2} & \text{if } L\not= 1 \text{ and }k \not= 2\\
\mathbb{F}^2 & \text{if } L=1 \text{ and }k = 2\\
0 & \text{if } L\not= 1 \text{ and }k = 2\\
\end{cases}
\end{equation*}
and
\begin{equation*}
H^1(U;\mathcal{F}) \cong \begin{cases}
\mathbb{F} & \text{if } L=1 \text{ and }k \not= 2\\
0 & \text{if } L\not= 1 \text{ and }k \not= 2\\
\mathbb{F} & \text{if } L=1 \text{ and }k = 2\\
0 & \text{if } L\not= 1 \text{ and }k = 2.\\
\end{cases}
\end{equation*}
The conclusion is that the Vietoris Mapping theorem applies for the loop to be collapsed if and only if it is nonresonant, that is, if $L\not= 1$.
\end{proof}

\subsection{Geometry extraction algorithm}
\label{geometry_extract_sec}

If the topology is known, then generically a single section of the
excitation sheaf contains all of the geometric information.  This
result is stronger than the results obtained by previous authors (like
Proposition \ref{topology_obstruct}) which requires knowledge of many
sections.  However, the sheaf-theoretic framework provides a local,
iterative mechanism for describing the geometric information in a
quantum graph.  The central idea is that in collapsing a spanning tree
in the graph, the cohomology (and hence the sections) of flow sheaves
is unchanged.  Under such a collapse, however, the global sections are
very easily tied to the metric structure of the graph.  By
sequentially ``undoing'' the edge collapses, edge endomorphisms are
determined one at a time until all are determined.

Operationally, enough information can be obtained by placing a
directional sensor at each vertex with degree not equal to 2.  Each
such sensor detects the incoming wave amplitude along each incident
edge.  Since the algorithm measures phase differences between points
in the graph, this requires that the sensors be synchronized, and take
their measurements simultaneously.  From a practical point of view,
loss in the graph (which we have neglected in this section) limits the
visibility of the signal sources.  As such, it is probably unecessary
to require synchronization over all sensors placed in a lossy graph.    

\begin{thm}
\label{geometry_extract_thm}
Suppose that the edge endomorphisms of a transmission line sheaf
$\mathcal{F}$ are algebraically independent (in the ring of
$\mathbb{F}$-endomorphisms).  In the case of quantum graphs,
this is equivalent to requiring that the edge lengths are
algebraically independent.  Then the edge endomorphisms are determined by any nonzero section of $\mathcal{F}$ when Kirchoff conditions \eqref{kirchoff} are used.
\end{thm}
As might be expected, the proof of this result is more interesting than its statement, and proceeds by an inductive computation.

\begin{proof}
{\bf Base case:} We assume that $X$ consists of a boquet of circles,
namely that there exists a single vertex $v$ in $X$; we determine all
of the edge auomorphisms in $\mathcal{F}$ from $s$.  For concreteness,
assume that there are $n$ closed loops and $m$ open edges in $X$.  Consider a connected open set $U$ which contains
the unique vertex and none of loops completely.  Then $\mathcal{F}(U)$
will be isomorphic to $\mathbb{F}^{2n+m}$, that is there are $2n+m$
incoming signals entering the vertex.  Some of these, of course, will
be related if we consider $\mathcal{F}(X)$.  Suppose without loss of
generality, that $d_1, d_2, ..., d_{2n-1}, d_{2n}$ are the values of the
section $s$ on the loops at $v$, and that $c_1, ..., c_m$ are
the values of the section $s$ on each open edge at $v$.  We can
further organize the loops so that $d_j$ and $d_{j+1}$ are the values
at either end of a loop, following Figure \ref{tlcoh_1_fig}.  Hence, we can use \eqref{kirchoff} to obtain
a set of $n$ equations, one for each loop
\begin{equation*}
d_{2j+1}=L_j \left(\frac{2}{2n+m} \left(\sum_{i=1}^mc_i+\sum_{i=1}^{2n}d_i\right)-d_{2j}\right),
\end{equation*}
which can easily be solved for the edge endomorphisms $L_j$.

{\bf Inductive step:} We assume that there is a graph $Y$ which can be
obtained from $X$ by collapsing $f:X\to Y$ a single edge $e$ with distinct
endpoints $v_1,v_2$, and that all edge endomorphisms of $f_* \mathcal{F}$ are
known.  Assume (see Figure \ref{tlcollapse_fig}):

\begin{itemize}
\item That the degree of $v_1$ is $m$ and that the degree of
$v_2$ and $n$,
\item That the edge endomorphism of $e$ is $L$,
\item That $c_0, c_1, ..., c_{m-1}$ are the values of the section $s$, that are incoming for $v_1$,
\item That $c_0$ is the value of $s$ coming into $v_1$ along $e$, and 
\item That $d_0, d_1, ..., d_{n-1}$ are the values of
$s$ incoming to $v_2$,
\end{itemize}
then \eqref{kirchoff} gives the following equations (see Figure \ref{tlcollapse_fig} for notation):
\begin{equation}
\label{geo_eqn}
b=\frac{2}{m}(c_0+...+c_{m-1})-c_0, \; a=\frac{2}{n}(d_0+...+d_{n-1})-d_0.
\end{equation}
Using the edge endomorphism $L$, we note that $Lb=d_0$ and
$c_0=L^{-1}a$, which permits each equation in \eqref{geo_eqn} to be
solved for $d_0$ or $L$.
\begin{equation*}
d_0=\frac{2L}{m}(c_0+...+c_{m-1})-Lc_0=\frac{nL}{2-n}c_0-\frac{2}{2-n}(d_1+...+d_{n-1}).
\end{equation*}

Since the graph $X$ is finite, it only remains to see that the
induction can be started by finding a sequence of trees $\emptyset=T_1
\subset T_2 \subset ... \subset T$ in which $T_i - T_{i-1}$ is a
single edge, and $T$ is a spanning tree for $X$.  The base case is
obtained by examining $f:X \to X/T$, and each induction step is
obtained by considering the collapse via $f_i:X/T_{i-1} \to X/T_i$.
In each case, these maps satisfy the hypotheses of Lemma \ref{tlcollapse}.

\end{proof}

\section{Discussion}
\label{discussion_sec}

In order to implement our algorithms into a viable sensing system, one
needs to address discretization issues.  For instance, the topology
extraction algorithm assumes that there are enough
receivers to discriminate whether a coverage region (or intersections
thereof) is disconnected.  This may require a very high
density of receivers in order to make this discrimination with
confidence, since it essentially amounts to measuring the distance
between clusters in signal space.  Although such discretization effects are
out of the scope of this article, it is useful to note that the amount
of loss, the distribution of edge lengths, and the operating frequency
play an important role in determining the necessary receiver density. 

As a related point, detecting the appropriate threshold to use for
the visibility regions is less clear when using discrete
receivers.  A rigorous approach to this problem might use the tools of
persistent homology \cite{Edelsbrunner_2008} to attempt to capture the topology of the
visibility regions, and select appropriate thresholds.

While the geometry extraction algorithm assumes discrete receivers,
their synchronization constitutes a major limitation to performance.
Indeed, the synchronization requirement is modulated by the mutual
visibility of the receivers, which depends on signal loss.  On
the other hand, the synchronization requirement may also be relaxed by
the use of active sensing, in which the signal sources and sensors are
colocated.  This method would suggest that instead of studying homogeneous
solutions, we instead study fundamental solutions.

Finally, in practical urban imaging applications, the graph model is
inaccurate for open areas, though in simulation sensible results can
be obtained \cite{GHR_2010}.  Indeed, one needs to generalize the theory
discussed here to handle higher-dimensional cellular spaces.  Many
difficulties arise from this generalization, not the least of which is
that the cohomology of the excitation sheaf becomes
infinite-dimensional.  The refinement algorithm as stated here no
longer works correctly, either, as connectedness no longer implies
contractibility of intersections.  Thresholding to obtain
contractibility should still work correctly, though verifying that it
does will require considerable effort.

\section{Acknowledgements}
The author wishes to thank Dr. Yuliy Baryshnikov for suggesting the connection of this work to quantum graphs, and Professors Robert Ghrist and Yasu Hiraoka
for interesting and valuable discussions on this project.  In
addition, Hank Owen's ``outsider'' input helped to refine the
exposition.

\bibliography{qg_bib}
\bibliographystyle{plain}

\end{document}